\documentclass{amsart}
\usepackage{tikz-cd}
\usepackage{tikz}
\usepackage[utf8]{inputenc}
\usetikzlibrary{matrix,arrows.meta}
\usepackage{amsmath}

\usetikzlibrary{calc}
\usepackage{caption}
\usepackage{longtable}
\usepackage{graphicx}
\usepackage{subcaption}
\usepackage{epsfig}
\usepackage{color}
\usepackage{array}
\usepackage{amsmath}
\usepackage{amsthm}
\usepackage{amssymb}
\usepackage{enumerate}
\usepackage{amscd}
\usepackage{color}
\usepackage{relsize}
\usepackage[bookmarksnumbered,bookmarksopen,pdfusetitle,unicode]{hyperref}
\usepackage[all,cmtip]{xy}
\usepackage{ctable}
\theoremstyle{plain}
\newtheorem{thm}{Theorem}[section]
\newtheorem{cor}[thm]{Corollary}

\newtheorem{prop}[thm]{Proposition}

\theoremstyle{definition}
\newtheorem{defn}[thm]{Definition}

\theoremstyle{remark}
\newtheorem{rem}[thm]{Remark}

\begin{document}
	
	\title{Minimality of a toric embedded resolution of singularities\\
		\tiny after Bouvier-Gonzalez-Sprinberg} 
	
	\author{B. Karaden\.{I}z \c Sen, C. Pl\'enat and M. Tosun}
	
	\subjclass[2020]{14B05, 14M25, 32S45}
	
	\thanks{This work is partly supported by the projects TUBITAK no.118F320  and PHC Bosphore no.42613UE}
	
	\keywords{Singularities, toric resolutions, jet schemes, Newton polyhedron, profile}

	\begin{abstract}
		This paper is devoted to construct a minimal toric embedded resolution of a rational singularity via jet schemes. The minimality is reached by extending the concept of the profile of a simplicial cone given in \cite{cg}. 
	\end{abstract}
	\maketitle
	
	\section{Introduction}

	\noindent	Let $X$ be a variety with the singular locus $Sing(X)$. By  \cite{hironaka}, it is known that $(X, Sing(X))$ admits a resolution, means that there exists a smooth variety $\tilde{X}$ and a proper birational map $\tilde{X}\rightarrow X$ which is an isomorphism over $X\setminus Sing(X)$. Later, in \cite{Nash},  Nash introduced the arc spaces $X_{\infty}=\{\gamma: Spec \ \mathbb{C}[t]\rightarrow X\}$ associated with $X$ which provides additional information about a resolution; he also conjectured that the number of irreducible components of $X_{\infty}^{Sing(X)}$ (the arcs passing through $Sing(X)$) is at most the number of {\it essential} irreducible components of the exceptional locus of a resolution. J. Fernandez de Bobadilla and M. Pe Pereiran proved in \cite{bobadilla-pe} that the equality is true for surfaces (see also \cite{dedo}), but there are counterexamples in higher dimensions, see for example \cite{Fernex2,IK,JK}.

	\noindent Therefore it makes sense to ask whether one can build a resolution of $X$ by means of  its arc spaces. One way to deal with it is to use the link between the arc and jet spaces of $X$ as the space of arcs  $X_{\infty}$ may be viewed as the limit of the jets schemes $X_m=\{\gamma_m: Spec\ \frac{\mathbb{C}[t]}{t^{m+1}}\rightarrow X\}$ \cite{Fernex}.  We get to the relationship between some irreducible components of jet schemes and divisorial valuations via the correspondence between some irreducible families of arcs (known as cylinders) passing through a subvariety $Y$ and divisorial valuations over $Y$  \cite{ELM}.  This raises the following problem:
	{\it Can one construct an embedded resolution of singularities of $X\subset \mathbb{C}^n$ from the irreducible components of the space $X_m^{Sing(X)}$ of jets centered at $Sing(X)$}?

	\noindent  In light of this, the authors in \cite{bhcm, hc}, (generalizing the dimension $1$ case in \cite{LMR}), construct a toric embedded resolution from the jet schemes for some surface singularities which are Newton non-degenerate in the sense of Kouchnirenko \cite{kouch} and get the following diagram: 
	\[
	\xymatrix{
		\pi^{-1}(X)\cap \tilde S_{\Sigma }=\tilde X  \ar[d]^{\pi } \ar[r]^{}  & \tilde S_{\Sigma } \ar[d]^{\pi_{\Sigma } } \\
		X  \ar[r]^{f} &  \mathbb{C}^n}
	\]
	
	\noindent  where $\tilde{S}_{\Sigma}$ represents the smooth toric variety obtained by a regular refinement $\Sigma $ of  the dual Newton polyhedron $DNP(f)$ of $X:\{f=0\}$ using the valuations associated to the irreducible components of some $m$-jets schemes. 
	\begin{rem} With preceding notation,  the strict transform of $\{f=0\}$ by $\pi_\Sigma$ is the Zariski closure of $(\pi_\Sigma)^{-1} (\mathbb{C}^3 \cap \{f=0\}).$ 
	\end{rem}
	
	\noindent Moreover, the following result indicates that $\tilde{X}=\pi^{-1}(X)\cap \tilde{S_{\Sigma}}$ is smooth.
	
	\begin{thm}\label{NND} \cite{AGS, O1, Varc}
		Let  $X\subset \mathbb{C}^3$ where $X:\{f=0\}$ is Newton non-degenerate in the sense of Kouchnirenko. Then the following properties are equivalent:\\
		1) A refinement $\Sigma$ of $DNP(f)$ is regular.\\
		2) The proper birational morphism $\mu_\Sigma:Z_\Sigma \longrightarrow \mathbb{C}^3$ is an embedded toric resolution of singularities of $X$ where $Z_\Sigma$ is the toric variety associated with $\Sigma$.
	\end{thm}
	
	\noindent  The goal of this article, following the spirit in \cite{cg}, is to show that there is a minimal toric embedded resolution when $X$ is a surface with rational singularities of multiplicity 3 (RTP-singularities for short) and to provide an algorithm to build it. The complete list of the minimal abstract resolution graphs of RTP-singularities is presented in  \cite{Artin} where the author gives a characterization of rational singularities via their minimal abstract resolution graphs and proved that the embedding dimension for a rational singularity equals $"multiplicity+1"$. The explicit equations defining RTP-singularities in $\mathbb{C}^4$ are due to G. N. Tyurina \cite{Tyurina}. Using some suitable projections of these equations, the authors in \cite{mag} obtained the hypersurfaces $X'\subset \mathbb{C}^3$ with $dim(Sing(X'))=1$ whose normalizations are the surfaces given in \cite{Tyurina} and, they showed that $X'$ is Newton non-degenerate in the sense of Kouchnirenko. These nonisolated forms of RTP-singularities are served in \cite{bhcm} to construct a toric embedded resolution via the jet schemes $X_m$ of RTP-singularities. But the question of minimality remained open because the abstract resolution obtained  in \cite{bhcm} was not itself minimal. Here we define the minimality of the resolution as below:
	
	\begin{defn}
		Let $\Sigma $ be a regular refinement of the $DNP(f)$ with vectors in some subset $G_\Sigma \subset \mathbb R^3$.  A minimal toric embedded resolution is a smooth toric variety obtained by $\Sigma $ if the abstract resolution has no $-1$ curve and 
		$G_\Sigma= \cup G_\sigma$ where $\sigma$'s are full dimensional cones in $\Sigma $ with 
		$$ G_\sigma=\{x\in \sigma\cap \mathbb{Z}^n\backslash{\{0\}}~~|~~ \forall n_1,n_2 \in  \sigma\cap 
		\mathbb{Z}^n, x=n_1+n_2 \Rightarrow n_1=0 \ or\ n_2=0 \}.$$ 
	\end{defn}
	\noindent Using the equations obtained in \cite{ACMZ}, we show the following:
	\begin{thm}
		There exists an equation giving the nonisolated form of an RTP-singularity such that 
		
		\noindent $i)$ its abstract resolution graph is minimal,
		
		\noindent $ii)$ the chosen irreducible components of the $m$-jets schemes are associated with vectors which provide an embedded toric resolution,
		
		\noindent $iii)$ those vectors are  in  $G_\Sigma$.\\
	\end{thm} 	
	
	\noindent This implies by $ii)$ and by the fact that the vectors in $G_\Sigma$ are always in any resolution, $G_\Sigma$ is exactly composed of these chosen vectors. We also show that:
	
	\begin{cor}
		The Hilbert basis of the $DNP(f)$ of an RTP-singularity gives a minimal toric embedded resolution.
	\end{cor}
	
	\noindent Sketch of the proof:
	
	\noindent$i)$ Using the equations given in \cite{ACMZ}, we obtain the minimal abstract graph via Oka's algorithm. 	
	
	\noindent $ii)$ Let $\mathcal{C}_m$ be an irreducible component of $X_m^{Sing(X)}$. Then ${\psi_m^a}^{-1}(\mathcal{C}_{m})$ is an irreducible cylinder in $\mathbb{C}^3_\infty$ (where ${\psi_m^a}: \mathbb{C}^3_\infty \longrightarrow \mathbb{C}^3_m$ is the truncation morphism associated with the ambient space $\mathbb{C}^3$). Let $\eta$ be the generic point of ${\psi_m^a}^{-1}(\mathcal{C}_{m})$. { By Corollary 2.6 in} \cite{ELM},  the map  $\nu_{\mathcal{C}_m}:\mathbb C[x,y,z]\longrightarrow \mathbb{N}$ defined by  
	$$ \nu_{\mathcal{C}_m}(h)=\mbox{ord}_t h\circ \eta $$
	is a divisorial valuation on $\mathbb C^3.$ We can associate a vector with $\mathcal{C}_m,$ called the weight vector, in the following way:
	$$v(\mathcal{C}_m):=(\nu_{\mathcal{C}_m}(x),\nu_{\mathcal{C}_m}(y),\nu_{\mathcal{C}_m}(z))   \in \mathbb{N}^{3}.$$ 
	We define the "good" irreducible components of jets schemes giving a resolution after computing the graph of the jet schemes (see \cite{bhcm, Mo, hc} for definition and detailed computations) and call the corresponding vectors as ''essential valuations''.
	
	\noindent $iii)$ Finally, to show that the essential valuations are in $G_{\Sigma}$, we introduce, following \cite{cg}  the profile for a cone generated by at least $3$ vectors. Then we show that the essential valuations are inside the profile; more precisely, we find a convex set inside the profile such that the vectors reach the hypersurfaces delimiting these sub-cones so-called sub-profiles. The convexity implies that the essential valuations are free over $\mathbb{Z}$, i.e. in  $G_{\Sigma}$. Thus as they give a non-singular refinement of $DNP(f)$, the essential valuations and elements of $G_{\sigma}$ (for each $\sigma$) coincide.
	\\
	
	\noindent Our remarks and questions: 
	\\
	\noindent $1)$ Question 1: It is known that the vectors obtained via tropical valuations of $X$ give the minimal abstract resolution of $X$ (see \cite{AGS, ar-hu}). We observe the intersection of the set of vectors in the Groebner fan of $X$ with the set of vectors obtained from jet schemes of $X$ is exactly the Hilbert basis for rational double point singularities (RDP-singularities). Is this true for all Newton non-degenerate singularities? 
	\\
	\noindent $2)$ Question 2: For RDP-singularities and RTP-singularities all the vectors in the Hilbert basis lie inside the profiles. Is the fact that the vectors in the Hilbert basis lie inside the profile a characterization of rational singularities? For example, the surfaces defined by $f=y^3+xz^2-x^4=0$ and $f=z^2+y^3+x^{21}=0$ have elliptic singularities and they are Newton non-degenerates. Their Hilbert basis give a resolution of singularities; but in both cases, the profile does not contain all the vectors in the Hilbert basis. 
	\\
	\noindent $3)$ Question 3: Does Hilbert basis give an embedded resolution for any Newton non-degenerate singularity? \\
	
	\noindent This article  is structured as follows: We start by recalling the definition of Hilbert basis of a cone. We generalize the notion of a profile given in \cite{cg}. Then, using the new equations of RTP-singularities (comparing with \cite{mag, bhcm}) we develop the proof of the theorem for $B$-types which was a special case in \cite{bhcm} as the authors did not obtain a toric embedded resolution. We end up with some remarks on the preceding questions. One can find in the Appendix the computations for the RTP-singularities.
	
	\section{Hilbert basis of polyhedral cones}
	
	\noindent Let $n,r\in \mathbb{N}^{*}$. Let $v_1,\ldots,v_r$ be some vectors in $\mathbb{Z}^{n}$. A rational polyhedral cone in $\mathbb{R}^n$ generated by the vectors $\{v_1,\ldots,v_r\}$ is the set 
	$$\sigma:=<v_1,\ldots,v_r>=\{v\in \mathbb{R}^n \mid \ v=\sum_{i=1}^{r}\lambda_i v_i, \ \lambda_{i}\in \mathbb{R}_{\geq 0}\}.$$ 
	When $\sigma$ doesn't contain any linear subspace of $\mathbb{R}^n$ we call it {\it strongly convex}. In the sequel, a cone will mean a strongly convex rational polyhedral cone.  The dimension of $\sigma$ is the dimension of the subspace $span\{v_1,\ldots,v_r\}$ in $\mathbb{R}^n$. Two cones  $\sigma $ and $\sigma' $ in $\mathbb{R}^n$  are said to be equivalent if $dim(\sigma )=dim(\sigma')$ and there exists a matrix $A\in GL_n(\mathbb Z)$ with $M(\sigma)=A\cdot M(\sigma')$ where $M(\sigma)$ denotes the matrix $[v_1 \ldots v_r]$. When $dim(\sigma )=n=r$ we say that $\sigma $ is a {\it simplicial cone}.
	
	\begin{defn} 
		A vector $v \in \mathbb{Z}^n$ is called primitive if all its coordinates are relatively prime. A cone $\sigma=<v_1,\ldots,v_r>\subset \mathbb{R}^n$ is called regular if the generating vectors are primitive and $M(\sigma)$ is unimodular.
	\end{defn}
	
	\noindent It is well known that the notion of regular cones is important in toric geometry, and in singularity theory a regular cone leads to a smooth toric variety. A regular cone can be constructed from a non-regular cone. Such a process is called {\it regular refinement}; it consists of a refinement of a cone into the subcones by some $n-1$ dimensional subspaces such that every subcone in the subdivision is regular. Let's recall a few concepts to provide a better definition of getting a regular refinement of a cone. Consider the set $S_{\sigma}:=\sigma\cap \mathbb{Z}^n$ which is a finitely generated semigroup with respect to the addition. For special $\sigma$'s there are several methods to find the set of generators of $S_{\sigma}$. One method comes from integer programming \cite{giles}.
	
	\begin {defn} A subset $H_{\sigma}\subset S_{\sigma}$ is called the Hilbert basis of $\sigma$ if any element $u\in S_{\sigma}$ can be written as a non-negative integer combination of the elements  in $H_{\sigma}$ and it is the smallest set of generators with respect to inclusion.
\end{defn}

\begin{prop} {\rm \cite{rosales}}
	Every cone admits a finite Hilbert basis. 
\end{prop}

\begin{prop}
	The Hilbert basis $H_{\sigma}$ is contained in the parallelepiped
	$$P_{\sigma}:=\{u\in \mathbb{Z}^n \mid u=\sum_{i=1}^{r}\lambda_i v_r , \ 0\leq\lambda_i \leq  1\}.$$ 
\end{prop}	

\begin{proof} It follows from the fact that any vector $u=\sum_{\lambda_i\geq 0}\lambda_iv_i\in \sigma $ can be written as $u=\sum_{i=1}^{r}(\lfloor {\lambda_i}\rfloor + {\lambda'_i})v_i$ where $\lfloor {\lambda_i}\rfloor$ is the integer part of $\lambda_i$ and $\sum_{i=1}^{r}{\lambda'_i}\in P_{\sigma }$. 
	
\end{proof}

\begin{defn}
	The first primitive vector lying on a $1$-dimensional subcone of $\sigma$ is called an extremal vector of $\sigma$.
\end{defn}

\begin{thm}
	Let $\sigma\subset \mathbb{R}^3$ be a cone. If an element $u\in S_{\sigma}$ is in $H_{\sigma}$ then  it is an extremal vector in any regular refinement of $\sigma$.
\end{thm}

\begin{proof}
	Let $\Sigma$ be a regular refinement of $\sigma$. Denote by $\tau_1,\tau_2,\ldots, \tau_k$ the maximal dimensional regular subcones in $\Sigma$. Let $u\in H_ {\sigma}$.  So $u$ belongs to at least one of $\tau_i$'s and $u=\alpha_1v_{1}^{(i)}+\alpha_2v_{2}^{(i)}+\alpha_3v_{3}^{(i)}\in \sigma$ where $v_{1}^{(i)}, v_{2}^{(i)}, v_{3}^{(i)}$ are the extremal elements of $\tau_i$, which is a basis for $\mathbb{Z}^3$. Since $u$ belongs to $H_\sigma$, we have $u=v_{j}^{(i)}$ for some $j=1,2,3$, which means that $u$ itself is an extremal vector for $\tau_i$.
	
\end{proof}

\noindent Let $\sigma=<v_1,v_2,\ldots,v_n>\subset \mathbb{R}^n$ be a simplicial cone. Consider the map
\begin{align*}
	l_{\sigma}: \mathbb{R}^n&\rightarrow \mathbb{Q}\\
	v&\mapsto l_{\sigma}(v)
\end{align*}
such that $\l_{\sigma}(v_i)=1$ with each extremal vector $v_i$ for $\sigma$. 

\begin{defn}
	\cite{cg} The subset $$p_{\sigma}:=\sigma\cap l_{\sigma}^{-1}([0;1])$$ is called the profile of $\sigma$.
\end{defn}

\noindent In the case $\sigma \subset \mathbb{R}^n$ is non-simplicial (which will be often the case for RTP-singularities below), we extend the definition as below.

\begin{defn}
	The profile of a cone $\sigma=<v_1,v_2,\ldots,v_r>\subset \mathbb{R}^n$ is the smallest convex hull such that its extremal vectors are exactly $v_1,v_2,\ldots,v_r$.
\end{defn}

\begin{rem}
	It may happen that all extremal vectors are on a unique hyperplane even though $\sigma=<v_1,v_2,\ldots,v_r>\subset \mathbb{R}^n$ is non-simplicial. In this case, $p_{\sigma}$ is defined as in the case of a simplicial cone.\\ Moreover, $p_{\sigma}$ can be identified with its boundaries composed by the union of at most $(r-2)$ hyperplanes in $\mathbb{R}^n$. 
\end{rem}

\begin{prop}
	Let $\sigma=<v_1,v_2,\ldots,v_r>\subset \mathbb{R}^n$. There is no other integer point in $p_{\sigma}$ than the elements of $H_{\sigma}$.
\end{prop}

\begin{proof}
	Assume that $r=n$ and $\sigma $ is simplicial. We have $v=\sum_{i=1}^{n}\alpha_{i}v_i \in \sigma$ with $\alpha_{i}\in \mathbb{R}_{\geq 0}$. Let $v\in p_{\sigma}$. Then $l_{\sigma}(v)\in [0,1]$ which means
	$$0\leq l_{\sigma}(\alpha_{1}v_1+\alpha_{2}v_+\ldots+\alpha_{n}v_n)=\alpha_{1}l_{\sigma}(v_1)+\alpha_{2}l_{\sigma}(v_2)+\ldots+\alpha_{n}l_{\sigma}(v_n)\leq 1$$
	\noindent Since $l_{\sigma}(v_i)=1$ for all $i$, we have $0\leq \alpha_{1}+\alpha_2+\ldots+\alpha_{n}\leq1$. If there exists one $i_0\in \{1,\ldots,n\} $ such that $\alpha_{i_0}=1$ we get $v=v_{i_0} \in p_{\sigma}$. If not, we have $\alpha_{i}=\frac{a_i}{b_i}$ with $a_i<b_i$, $b_i\neq 0$ for all $i$ . As $v$ cannot be written as the sum of two integer vectors we have $v\in H_{\sigma}$.
	
	\noindent When $\sigma$ is a non-simplicial cone, we get the affirmation by applying the discussion above to the each simplicial subcone lying in a suitable regular refinement of $\sigma$ into simplicial cones.
	
\end{proof}

\section{The new equations for RTP-singularities}

\noindent Let $X$ be defined by a complex analytic function 
$$f(z_1,z_2,\ldots z_n)=\sum_{(a_1,a_2,\ldots a_n)\in \mathbb Z^{n}_{\geq 0}} c_{(a_1,a_2,\ldots a_n)}{z_1}^{a_1}{z_2}^{a_2}\ldots {z_n}^{a_n}$$ 
The closure in $\mathbb R^n$ of the convex hull of the set 
$$S(f):=\{(a_1,a_2,\ldots a_n)\in {\mathbb Z^n}_{\geq 0}\mid c_{(a_1,a_2,\ldots a_n)}\not =0\}$$
is called {\it the Newton polyhedron} of $f$, denoted as $NP(f)$. Let $\Sigma(f)$ be a regular refinement of the dual Newton polyhedron $DNP(f)$. Then $X_{\Sigma(f)}$ is smooth and a toric map $X_{\Sigma(f)}\rightarrow \mathbb{C}^n $ obtained between the corresponding toric varieties is a toric embedded resolution of $X$ (see \cite{AGS,khovanskii,O1,Varc}). When the coefficients $c_{(a_1,a_2,\ldots a_n)}\in \mathbb C$ are generic and the $NP(f)$ is nearly convenient we say that $f$ is non-degenerate with respect to $NP(f)$. In this case, the regular refinement $\Sigma_2(f)$ of all $2$-dimensional cones in $DNP(f)$ gives an abstract resolution of $X$ and it induces a toric embedded resolution of $X$ by getting a regular refinement $\Sigma_3(f)$ of all $3$-dimensional cones in $\mathbb R^3$ \cite{oka,Varc}.  
\begin{defn}
	Such an embedded resolution is said to be minimal if the vectors appearing in the regular refinement are all irreducible and if the abstract resolution does not present $-1$ curves.
\end{defn}

\noindent We present below an algorithm to find a minimal toric embedded resolution of RTP-singularities which are treated in \cite{mag, bhcm}. Here we use the equations obtained in \cite{ACMZ} to present the non-isolated form of RTP-singularities different than those in \cite{mag}.

\begin{thm}
	The regular refinement $\Sigma_2(f)$ of all $2$-dimensional cones in $DNP(f)$ where $f$ is one of the following equations gives the minimal abstract resolution of the corresponding $RTP$-singularity. 
	
	$i)$ $\bf A_{k,l,m}:$
	
	\ \ \ \ $\bullet$ For $k=l=m>1$
	$$y^{3m+3}+xy^{m+1}z-xz^2-z^3=0$$	
	
	\ \ \ \ $\bullet$ For $k=l<m$ and $k,l,m\geq 1$
	$$y^{k+l+m+3}+y^{2k+2}z+y^{k+1}z^2+xy^{k+1}z+xz^2-z^3=0$$
	
	\ \ \ \ $\bullet$ For  $l<m<k$ and $k,l,m\geq 1$
	
	\ \ \ \ \ \ \ \ \	 $l+k>2m$ and $l+k\leq 2m,$ $l+k$ is even 
	$$y^{3k}+y^{2k+m+l-2}-2y^{l+k}z-xy^kz+y^mz^2+xz^2-z^3=0$$
	
	\ \ \ \ $\bullet$ For  $l<m<k$ and $k,l,m\geq 1$
	
	\ \ \ \ \ \ \ \ \	 $l+k\leq 2m,$ $l+k$ is odd
	$$y^{2k+m}+y^{k+m}z+y^{l+k}z+xy^{k}z-y^kz^2+y^{l}z^2+xz^2-z^3=0$$
	
	$ii)$ $\bf B_{k,n}:$ For  $r\geq 1$, $n\geq 2$
	
	\ \ \ \ $\bullet$ For  $k=2r-1$
	$$x^{2n+3}z-x^ry^2-y^2z=0$$
	
	\ \ \ \ $\bullet$ For  $k=2r$
	$$x^{n+r+2}y-x^{2n+3}z+y^2z=0$$
	
	$iii)$ $\bf C_{n,m}:$ For  $n\geq 3$, $m\geq 2$
	$$x^{n-1}y^{2m+2}+y^{2m+4}-xz^2=0$$

	$iv)$ $\bf D_n:$ For  $n\geq 1$
	$$x^{2n+2}y^2-x^{n+3}z+yz^2=0$$
	
	$v)$ $\bf E_{60}:$
	$$z^3+y^3z+x^2y^2=0$$
	
	$vi)$ $\bf E_{07}:$
	$$z^3+y^5+x^2y^2=0$$
	
	$vii)$ $\bf E_{70}:$
	$$z^3+x^2yz+y^4=0$$
	
	$viii)$ $\bf F_{k-1}:$ For $k\geq 2$
	$$y^{2k+3}+x^2y^{2k}-xz^2=0$$
	
	$ix)$ $\bf H_n:$ For  $n\geq 1$
	
	\ \ \ \ $\bullet$ For  $n=3k-1$
	$$z^3+x^2y(x+y^{k-1})=0$$
	
	\ \ \ \ $\bullet$ For  $n=3k$
	$$z^3+xy^kz+x^3y=0$$
	
	\ \ \ \ $\bullet$ For  $n=3k+1$
	$$z^3+xy^{k+1}z+x^3y^2=0$$
	
\end{thm}

\noindent  Recall that, when $X\subset \mathbb C^n$ is a surface with a rational singularity, the minimality of an abstract resolution is characterized by the fact that there is no $-1$ curve in the resolution. These new equations are Newton non-degenerate in the sense of Kouchnirenko, so one can show  by the Oka's algorithm \cite{oka} that the abstract resolution in each case is minimal (see the tables in the Appendix). Note that the equations given in \cite{mag} for the types $E$'s and $H_n$ are the same as the one given above and lead us to the minimal abstract resolution, which is not the case for the other types with the equations presented in  \cite{mag}.

\section{Minimal toric embedded resolutions: The $B_{k,n}$-singularities}
\subsection{Jet schemes and embedded valuations}

\noindent Let us recall few facts about the jet schemes and define the set $EV(X)$ of the embedded valuations, that will provide us the regular refinement of a given $DNP(f)$. Let $X\in \mathbb{C}^3$ be an hypersurface defined by one of the equations above. Let $m\in \mathbb{N}$. Consider the morphism
$$\varphi:\frac{\mathbb{C}[x,y,z]}{<f>} \to \frac{\mathbb{C}[t]}{<t^{m+1}>} $$

\noindent where  $x(t)=x_{0}+x_{1}t+x_{2}t^2+\ldots+x_{m}t^m$ \ \ (mod $t^{m+1})$

\ \ \ \ \ $y(t)=y_{0}+y_{1}t+y_{2}t^2+\ldots+y_{m}t^m$ \ \ \ (mod $t^{m+1})$

\ \ \ \ \ $z(t)=z_{0}+z_{1}t+z_{2}t^2+\ldots+z_{m}t^m$ \ \ \  (mod $t^{m+1})$

\ \

\noindent such that $f(x(t), y(t), z(t))=F_0+tF_1+\ldots+t^mF_m$ (mod $t^{m+1})$. The $m$-th jets scheme of $X$ is defined by
$$X_m=Spec(\frac{\mathbb{C}[x_i, y_i, z_i; \ \ i=1,\ldots,m]}{<F_0, F_1,\dots ,F_m>})$$
\noindent It is a finite dimensional scheme. For $n\in \mathbb{N}$ with $m>n$ we have a canonical projection $\pi_{m,n}: X_m \to X_n$. These affine morphisms verify $\pi_{m,p}\circ \pi_{q,m}=\pi_{q,p}$
for $p<m<q$ and they define a projective system whose limit is a scheme that we denote $X_{\infty}$, which is called the arcs space of $X$. Note that $X_0=X$. The canonical projection $\pi_{m,0}:X_m\longrightarrow X_0$ will be denoted by $\pi_{m}$. Denote also $X_{m}^{Y}:=\pi_m^{-1}(Y)$ for $Y\subset X$.  Consider the canonical morphism $\Psi_m : X_{\infty}\longrightarrow X_m$ and the truncation map ${\psi_m^a}: \mathbb{C}^3_\infty \longrightarrow \mathbb{C}^3_m$ associated with the ambient space $\mathbb{C}^3$, {{here the exponent $"a"$ stands for \it{ambient map} }}. The morphism ${\psi_m^a}$ is a trivial fibration, hence
${\psi_m^a}^{-1}(\mathcal{C}_{m})$ is an irreducible cylinder in $\mathbb{C}^3_\infty.$ Let $\eta$ be the generic point of ${\psi_m^a}^{-1}(\mathcal{C}_{m})$. { By Corollary 2.6 in} \cite{ELM},  the map  $\nu_{\mathcal{C}_m}:\mathbb C[x,y,z]\longrightarrow \mathbb{N}$ defined by  
$$ \nu_{\mathcal{C}_m}(h)=\mbox{ord}_t h\circ \eta $$
is a divisorial valuation on $\mathbb C^3.$ To each irreducible component $\mathcal{C}_m$ of $X_m^{Y}$, let us  associate a vector, called the weight vector, in the following way:
$$v(\mathcal{C}_m):=(\nu_{\mathcal{C}_m}(x),\nu_{\mathcal{C}_m}(y),\nu_{\mathcal{C}_m}(z))   \in \mathbb{N}^{3}.$$ 
\noindent Now, we want to characterize the irreducible components of $X_m^{Y}$ that will allow us to construct an embedded resolution of $X$.
For $p\in \mathbb{N},$ we consider the following cylinder in the arcs space:
$$Cont^p(f)=\{\gamma \in \mathbb{C}^3_\infty : \mbox{ord}_tf\circ \gamma=p\}.$$

\begin{defn} Let $X:\{f=0\}\subset \mathbb{C}^3$ be a surface. Let $Y$ be a subvariety of $X$.\\
	$(i)$ The elements of the set:
	$$EC(X):=\{\text{Irreducible components }\mathcal{C}_m \ \text{of} \ X_m^{Y} \ \text{such that }  {\psi_m^a}^{-1}(\mathcal{C}_{m})\cap Cont^{m+1}{f} \not = \emptyset$$
	$$\ \ \ \ \ \ \ \  \ \ \ \ \ \ \ \ \ \ \ \ \  \ \ \ \ \text{and } v(\mathcal{C}_m)\not=v(\mathcal{C}_{m-1})~~\text{for {any component}} ~~\mathcal{C}_{m-1} ~~\text{verifying}$$
	$$~~ 
	\pi_{m,m-1}(\mathcal{C}_m)\subset \mathcal{C}_{m-1}, m\geqslant 1 \}$$
	are called the \emph{essential components} for $X$. 
	
	\noindent $(ii)$ The elements of the set of associated valuations
	$$EV(X):=\{\nu_{\mathcal{C}_m},~~\mathcal{C}_m\in EC(X) \}$$
	are called \emph{embedded valuations} for $X$. 
\end{defn}

\noindent In \cite{bhcm} the authors explicitly construct the jet graphs and embedded resolutions for all cases of RTP-singularities; but the abstract resolutions of the singularities of types $A, B, C, D$ and $F$ were containing at least one curve with self-intersection $-1$ which is not the case for the new equations. Moreover, the equation of $B$-type singularities given in \cite{bhcm} is very particular since its jet graph provides  a resolution which is not a refinement of the $DNP(f)$. In this article, we find a toric embedded resolution with the help of the jet graph of the new equation for $B$-type singularities.  We also show that the vectors obtained from the jets are irreducible by showing that they are inside the profile, more exactly they reach  hypersurfaces that form a new convex subcone inside the profiles, that we call {\it subprofils}; for geometrical reason, the vectors will be in $G_\sigma$ for each cone $\sigma$.
In the sequel we present the entire computations for $B$-type singularities, the results for the other cases are collected in a table (see Appendix).

\subsection{$B_{k,n}$-singularities}

Consider the hypersurface $X\subset \mathbb{C}^3$ having $B_{k,n}$ singularities, means its defining equation is $f=x^{2n+3}z-x^{r}y^2-y^2z=0$ for $k=2r-1$ or $f=x^{n+r+2}y-x^{2n+3}z+y^2z=0$ for $k=2r$ (given in the list above). 

\begin{rem}
	Comparing with \cite{mag, bhcm}, we see that we only have two cases to treat instead of five cases. Moreover the computation process is simpler since, in both cases the $NP(f)$ admits a unique compact face. 
\end{rem}

\noindent The $DNP(f)$ for $k=2r-1$ and $k=2r$ are as follows: 
\begin{figure}[h]
	\centering
	\includegraphics[width=0.9\textwidth]{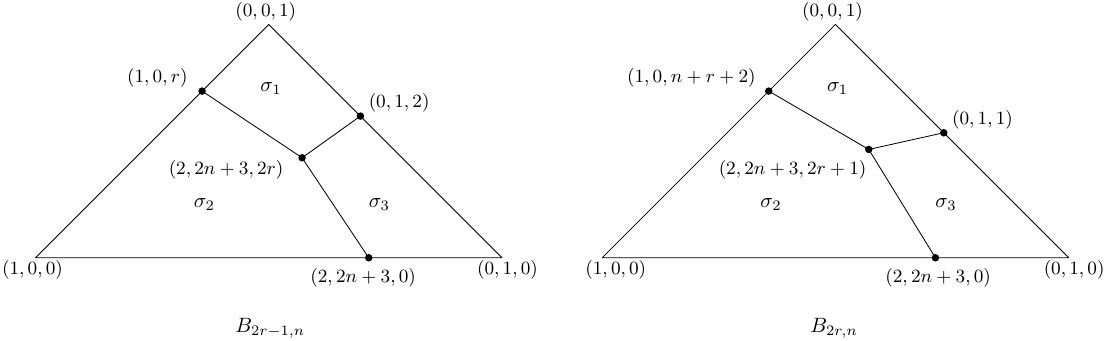}
	\caption{$DNP(f)$ of $B_{k,n}$-singularities}
\end{figure}

\begin{thm} \label{thm1} For $B_{2r-1,n}$-singularities, the embedded valuations of $X$ are
	
	$\bullet$ $(1,0,1), (1,0,2), \ldots, (1,0,r)$
	
	$\bullet$ $(2,2n+3,0),(2,2n+3,1),\ldots,(2,2n+3,2r)$
	
	$\bullet$ $(0,1,1),(0,1,2), (1,n+2,r+1)$
	
	$\bullet$ $(1,s,0),(1,s,1),\ldots,(1,s,r)$ with $1\leq s \leq n+2$
	
	\noindent and,  for $B_{2r,n}$-singularities, the embedded valuations of $X$ are
	
	$\bullet$ $(1,0,1), (1,0,2), \ldots, (1,0,n+r+2)$
	
	$\bullet$ $(2,2n+3,0),(2,2n+3,1),\ldots,(2,2n+3,2r+1)$
	
	$\bullet$  $(0,1,1), (1,n+2,r+1)$
	
	$\bullet$ $(1,1,0), (1,1,1),\ldots,(1,1,n+r+1)$
	
	$\bullet$ $(1,2,0), (1,2,1),\ldots,(1,2,n+r)$
	
	\ \ \ \ \ \ \ $\vdots$
	
	$\bullet$ $(1,n+2,0),(1,n+2,1),\ldots,(1,n+2,r).$\\
	
	\noindent In both cases the embedded valuations give a toric embedded resolution of $X$ and the vectors on the skeleton gives the minimal abstract resolution graph of the singularity.
\end{thm}

\begin{proof}
	In order to give the elements of $EV(B_{k,n})$, we compute the jet graph of the singularity as in \cite{bhcm,hc}. The jet graph of $B_{2r-1,n}$-singularities is 
	\begin{figure}[h]
		\centering
		\includegraphics[width=0.7\textwidth]{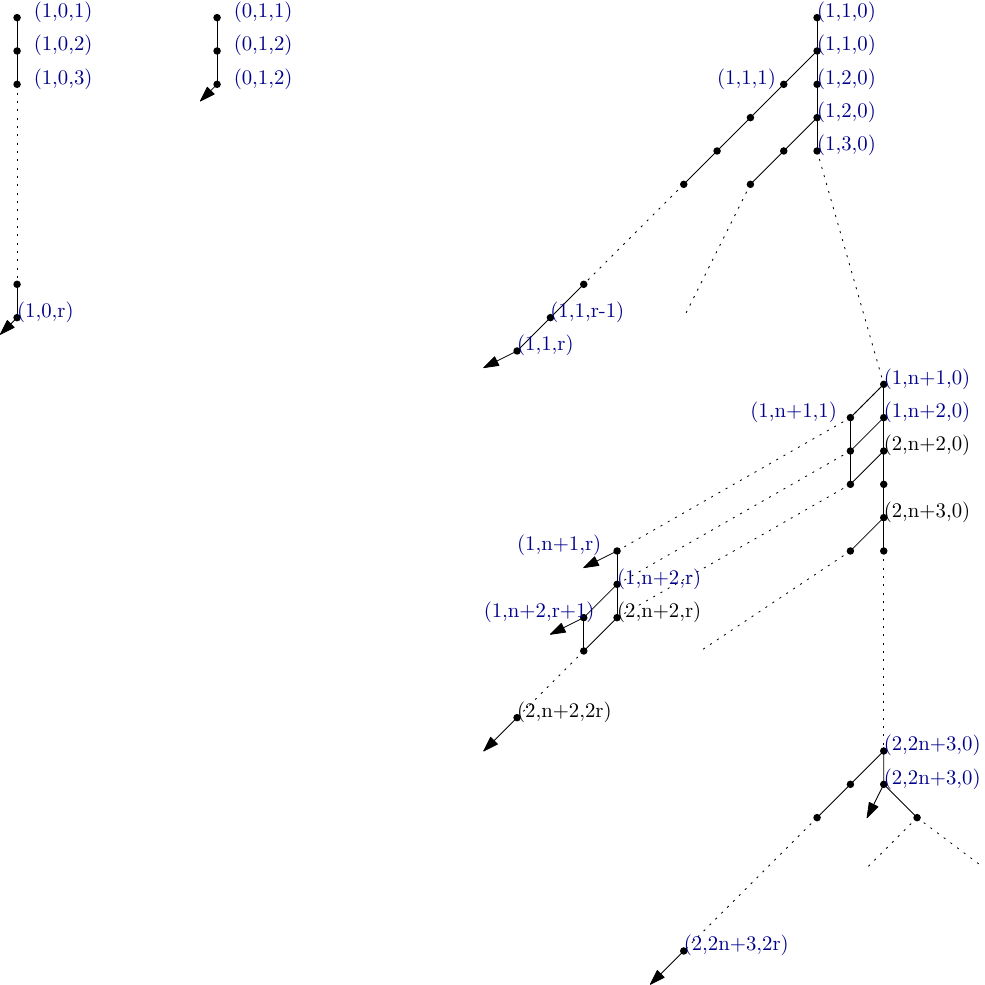}, 
		\caption{Jet Graph of $B_{2r-1,n}$-singularities}
	\end{figure}
	
	\newpage
	
	\noindent and the jet graph of $B_{2r-1,n}$-singularities is
	
	\begin{figure}[h]
		\centering
		\includegraphics[width=0.7\textwidth]{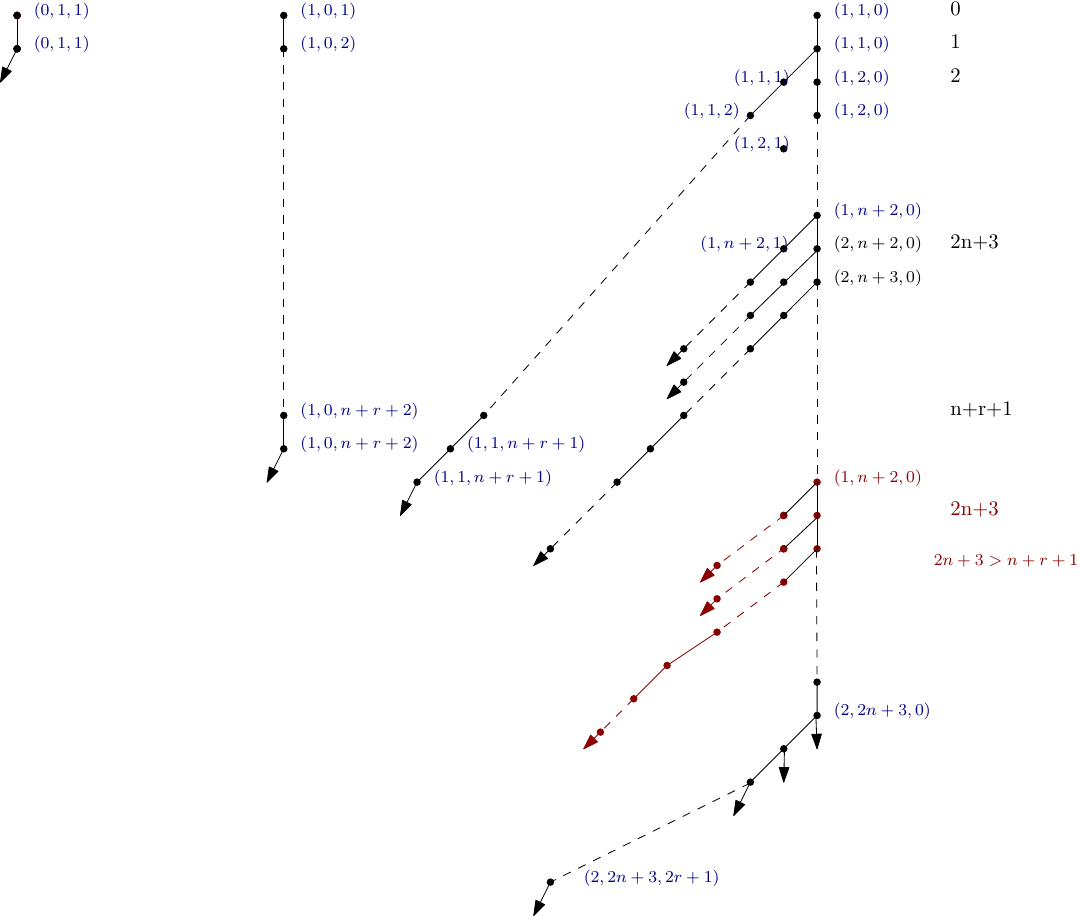}
		\caption{Jet Graph of $B_{2r,n}$-singularities}
	\end{figure}
	
	\noindent The vectors in the set $EV(B_{k,n})$ gives a regular refinement of the $DNP(f)$. They are the vectors written in blue in the jet graphs. A (simplicial) regular refinement of each subcone in the $DNP(f)$ for $B_{2r-1,n}$-singularities with these elements is illustrated in the following figure:
	\begin{figure}[h]
		\centering
		\includegraphics[width=0.7\textwidth]{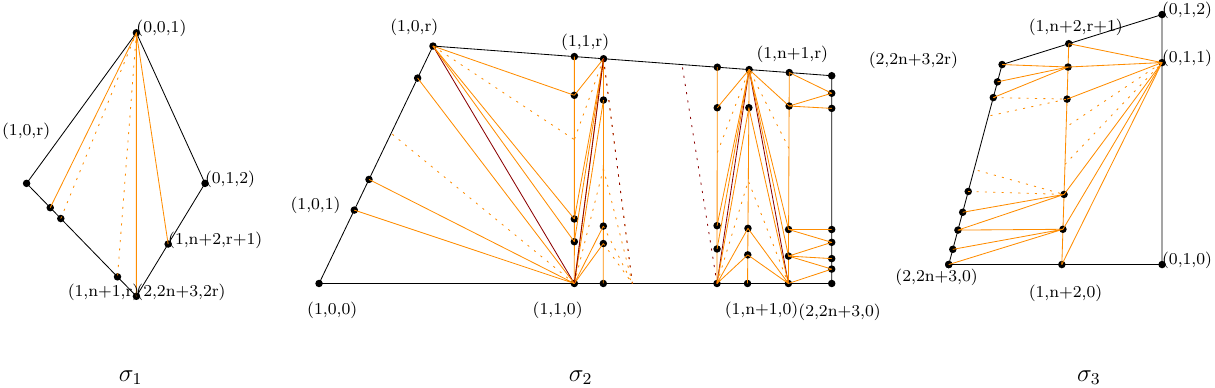}
		\caption{Resolution of $B_{2r-1,n}$-singularities}
	\end{figure}
	
	\noindent The refinement of $\sigma_1$ in $DNP(f)$ is regular since we have 
	{\tiny $\left \vert \begin{array}{ccc}
			\  0&1&1\\
			\ 0&s&s+1\\
			\ 1&r&r\\
		\end{array}\right \vert=1$} for $0\le s\le n$ and  also 	{\tiny
		$\left \vert \begin{array}{ccc}
			0&0&1\\
			0&1&n+2\\
			1&2&r+1\\
		\end{array}\right \vert=\left \vert \begin{array}{ccc}
			0&2&1\\
			0&2n+3&n+2\\
			1&2r&r+1\\
		\end{array}\right \vert=\left \vert \begin{array}{ccc}
			\  0&2&1\\
			\ 0&2n+3&n+1\\
			\ 1&2r&r\\
		\end{array}\right \vert=1$.}

	\noindent For the regularity of $\sigma_2$ in $DNP(f)$, we look at two subcones:
	
	\noindent For $<(1,n+1,0),(1,n+1,r),(2,2n+3,2r),(2,2n+3,0)>$,    
	{\tiny
		$\left \vert \begin{array}{ccc}
			2&2&1\\
			2n+3&n+1&n+1\\
			2s+1&s+1&s\\
		\end{array}\right \vert=1$, $\left \vert \begin{array}{ccc}
			2&2&1\\
			2n+3&2n+3&n+1\\
			2s&2s+1&s\\
		\end{array}\right \vert=1,$ $\left \vert \begin{array}{ccc}
			2&2&1\\
			2n+3&2n+3&n+1\\
			2s&2s-1&s\\
		\end{array}\right \vert=1$} for $0\le s\le r-1$.

	\noindent And, for the subcone $<(1,n+1,0),(1,n+1,r),(1,0,r),(1,0,0)>$ we have\\
	
	{\tiny $\left \vert \begin{array}{ccc}
			\  1&1&1\\
			\ k&k&k+1\\
			\ l&l+1&r\\
		\end{array}\right \vert=1$} for $0\le l\le r$,  	{\tiny $\left \vert \begin{array}{ccc}
			\  1&1&1\\
			\ k&k&k-1\\
			\ l&l+1&r\\
		\end{array}\right \vert=1$}  for $0\le l\le r$,

	{\tiny $\left \vert \begin{array}{ccc}
			\  1&1&1\\
			\ k&k&k+1\\
			\ l&l+1&0\\
		\end{array}\right \vert=1$} for $0\le l\le r$, 	{\tiny $\left \vert \begin{array}{ccc}
			\  1&1&1\\
			\ k&k&k-1\\
			\ l&l+1&0\\
		\end{array}\right \vert=1$} for $0\le l\le r$.   \\\\
	
	\noindent {Finally for the regularity of $\sigma_3$ in $DNP(f)$,} we look at the subcone $<(1,n+2,0),(1,n+2,r+1),(2,2n+3,2r),(2,2n+3,0)>$ for which we have, for all $0\le s\le r-1$\\
	
	\noindent {\tiny $\left \vert \begin{array}{ccc}
			\  2&2&1\\
			\ 2n+3&2n+3&n+2\\
			\ 2s&2s+1&s\\
		\end{array}\right \vert=1$,  $\left \vert \begin{array}{ccc}
			\  2&2&1\\
			\ 2n+3&2n+3&n+2\\
			\ 2s&2s-1&s\\
		\end{array}\right \vert=1$} and {\tiny $\left \vert \begin{array}{ccc}
			\  2&2&1\\
			\ 2n+3&n+2&n+2\\
			\ 2s+1&s+1&s\\
		\end{array}\right \vert=1$}.\\
	
	\noindent and the subcone $<(1,n+2,0),(1,n+2,r+1),(0,1,2),(0,1,0))>$ has, for $0\le l\le r$\\
	
	{\tiny $\left \vert \begin{array}{ccc}
			\  1&1&0\\
			\ n+2&n+2&1\\
			\ l&l+1&1\\
		\end{array}\right \vert=1$}. \\
	
	\noindent Hence $DNP(f)=\sigma_1\cup \sigma_2\cup \sigma_3$ is regular. A similar computation gives a regular refinement for the $B_{2r,n}$-singularities. Using Oka's algorithm, we can compute self-intersections and genus of the corresponding curves, and show that we get the minimal abstract resolution.
	
\end{proof}

\begin{thm}
	\label{thm2}
	The vectors in $EV(B_{k,n})$ lives inside the profiles of $B_{k,n}$ singularities. More precisely, for each subcones in $DNP(f)$ there exists  hypersurfaces inside each profile which is reached by the vectors in $EV(B_{k,n})$. Moreover the vectors in each subcones are free over $\mathbb{Z}$.
\end{thm}

\begin{proof} For $B_{2r-1,n}$-singularities, let's look at the 3-dimensional subcones in $DNP(f)$:\\
	
	\noindent  {\bf For $\sigma_1=<(0,0,1),(1,0,r),(0,1,2),(2,2n+3,2r)>$,} the profile $p_{\sigma_1}$ is bounded by two hyperplanes which are
	$$H_1:(2n-2nr+3-3r)-y+(2n+3)z-(2n+3)=0\ \ {\rm and } \ \  H_2:(n-r+2)x-y+z-1=0$$
	Let $p_{\sigma_1}^1$ and $p_{\sigma_1}^2$ denote two cones bounded respectively by the hyperplanes $H_1^{(1)}:(r-1)x-z+1=0$ and $H_2^{(1)}:(n-r+2)x-y+z-1=0$. They form a convex hull inside the profile $p_{\sigma_1}$; we call them  (and by abuse of language, the hypersufaces too) {\bf subprofiles}. The coordinates of each vector in the set $\{(1,n+2,r+1), (1,1,n+r+1), (1,2,n+r), (1,3,n+r-1),\ldots,(1,n,r+2),(1,n+1,r+1)\}$ satisfies at least one of the equations defining $H_1^{(1)}$ and $H_2^{(1)}.$ Moreover  $p_{\sigma_1}^1\cup p_{\sigma_1}^2$ is convex. This implies that all the elements in the previous set are in $H_{\sigma_1}$.\\
	
	\noindent {\bf For $\sigma_2=<(1,0,0),(1,0,r),(2,2n+3,0),(2,2n+3,2r)>$,} the profile $p_{\sigma_2}$ is bounded by a unique hyperplane which is $H:(2n+3)x-y-(2n+3)=0$; it contains the vectors $(2,2n+3,1),(2,2n+3,2),\ldots,(2,2n+3,2r),(1,0,1),(1,0,2),\ldots,(1,0,n+r+1),(1,1,0),(1,1,1),(1,1,2),(1,1,3),\ldots,(1,1,n+r+1),(1,2,0),(1,2,1),\ldots,(1,2,n+r),(1,3,0),(1,3,1),\ldots,(1,3,n+r-1),\ldots,(1,n+1,0),(1,n+1,1),\ldots,(1,n+1,r+1)$. All these vectors including the generators are in the subprofile defined by two hyperplanes $H_{1}^{(2)}: x=1$ and $H_{2}^{(2)}: (n+2)x-y-1=0$.\\		
	
	\noindent {\bf For $\sigma_3=<(0,1,0),(0,1,2),(2,2n+3,0),(2,2n+3,2r)>$,} the profile $p_{\sigma_3}$ is bounded by a unique hyperplane $H: (n+1)x-y+1=0$; it contains the vectors $(2,2n+3,1),(2,2n+3,2),\ldots , (2,2n+3,2r),(1,n+2,0),(1,n+2,1),\ldots,(1,n+2,r+1) $. All these vectors including the generators belong to do subprofile defined by the hyperplane $H:(n+1)x-y+1=0$ (here profile and subprofile are the same).
	
	\begin{figure}[h]
		\centering
		\includegraphics[width=0.9\textwidth]{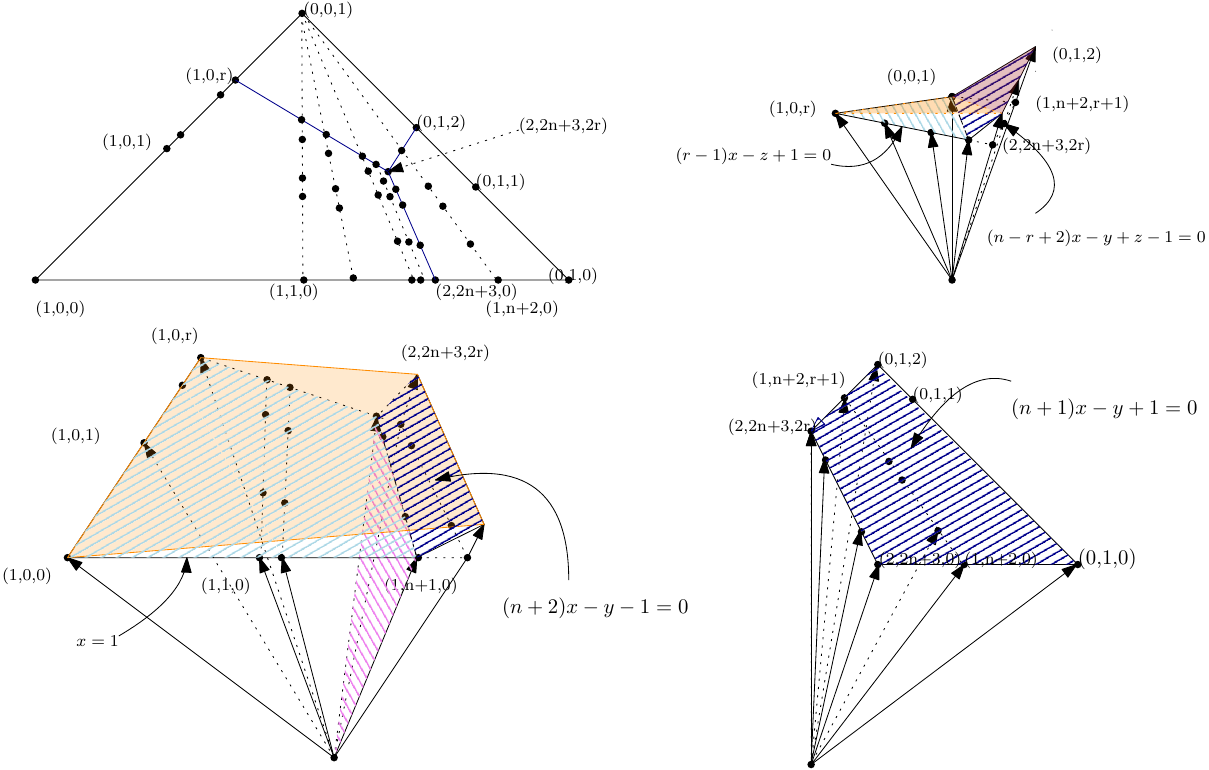}
		\caption{Profiles and subprofiles of $B_{2r-1,n}$-singularities}
	\end{figure}
	
	\noindent {\bf For $B_{2r,n}$-singularities}, the $DNP(f)$ and the $3$-dimensional subcones in it behave as in the following:\\
	
	\noindent  {\bf For $\sigma_1=<(0,0,1),(1,0,n+r+2),(0,1,1),(2,2n+3,2r+1)>$,} the profile $p_{\sigma_1}$ is bounded by two hyperplanes $H_1: (2n^2+2nr+5n+3r+3)x-(2n+2)y-(2n+3)z+(2n+3)=0$ and $H_2: rx-z+1=0$ (see figure below). It contains the vectors $(1,n+2,r+1), (1,1,n+r+2),(1,2,n+r),(1,3,n+r-1),\ldots,(1,n,r+2),(1,n+1,r+1)$. All these vectors including the generators are in the subprofile defined by the hyperplanes $H_{1}^{(1)}: (n^2+nr+2n+r+1)x-ny-(n+1)z+(n+1)=0$ and $H_{2}^{(1)}: rx-z+1=0$.\\
	
	\noindent {\bf For $\sigma_2=<(1,0,0),(1,0,n+r+2),(2,2n+3,0),(2,2n+3,2r+1)>$,} the profile $p_{\sigma_2}$ is bounded by the unique hyperplane $H: (2n+3)x-y-(2n+3)=0$. It contains the vectors 	$(2,2n+3,1), (2,2n+3,2),\ldots ,(2,2n+3,2r),(1,0,1), (1,0,2),\ldots ,(1,0,n+r+1), (1,1,0), (1,1,1), (1,1,2),(1,1,3),\ldots (1,1,n+r+1), (1,2,0),(1,2,1), \ldots , (1,2,n+r), (1,3,0), (1,3,1),\ldots,(1,3,n+r-1),\ldots , (1,n+1,0), (1,n+1,1),\ldots , (1,n+1,r+1)$ as all these vectors including the generators are in the subprofile defined by two byperplanes $H_1^{(2)}: x=1$ and $H_{2}^{(2)}: (n+2)x-y-1=0$.\\
	
	\noindent {\bf For $\sigma_3=<(0,1,0),(0,1,1),(2,2n+3,0),(2,2n+3,2r+1)>$,} the profile $p_{\sigma_3}$ is bounded by the unique hyperplane $H: (n+1)x-y+1=0$. It contains the vectors $(2,2n+3,1),(2,2n+3,2),\ldots,(2,2n+3,2r),(1,n+2,0),(1,n+2,1),\ldots,(1,n+2,r+1)$. All these vectors including the generators are in the subprofile defined by the hyperplane $H: (n+1)x-y+1=0$.
	
	\begin{figure}[h]
		\centering
		\includegraphics[width=0.9\textwidth]{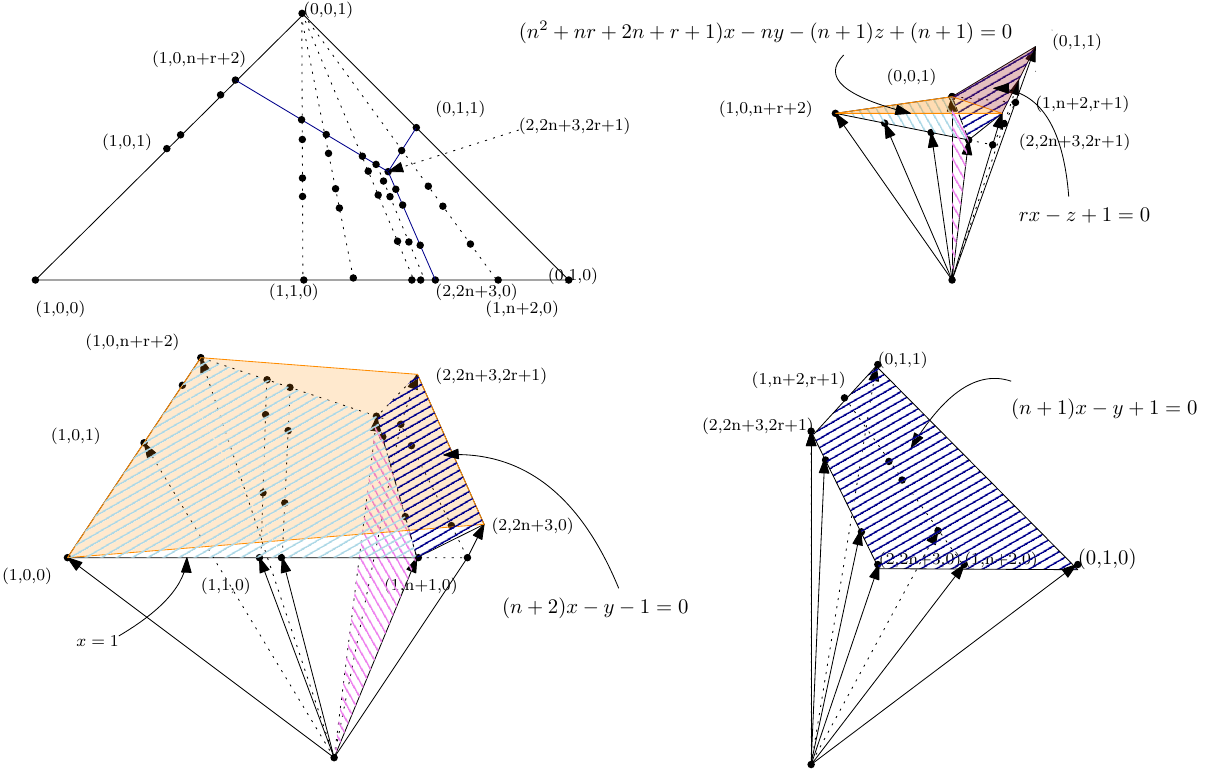}
		\caption{Profiles and subprofiles of $B_{2r,n}$-singularities}
	\end{figure}
\end{proof}

\begin{cor}
	Let $H_{DNP(f)}=H_{\sigma_1}\cup H_{\sigma_2}\cup H_{\sigma_3}$ be the Hilbert basis of $DNP(f)$. The elements of $EV(B_{k,n})$ are in $H_{DNP(f)}$ and give a minimal toric embedded resolution of the singularity.
\end{cor}

\noindent In fact by \ref{thm2}, they are irreducible and by \ref{thm1}, the elements give a resolution and form exactly the Hilbert basis of $DNP(f)$. In other words;

\begin{cor}
	For a $B_{k,n}$-singularity with its new equation, the union of Hilbert basis of each full dimensional subcone in $DNP(f)$ is the resolution of the singularity.
\end{cor}

\noindent For  all other RTP-singularities, we present the results in a table format (equations, subprofiles, vectors) 
in Appendix.

\begin{rem}
	For RDP-singularities, the profiles and subprofiles coincide (see  \cite{hc}).
\end{rem}

\section{Remarks on hypersurfaces with elliptic singularities}

\noindent Three natural questions arise from our algorithm applied in the previous sections:\\

\noindent $1)$ Does Hilbert basis give a toric embedded resolution for any Newton non-degenerate singularity?\\
$2)$ Let $\sigma $ be a 3-dimensional cone in $DNP(f)$.\\
$(a)$ Is it true for all rational singularities that each element in $H_{\sigma }$ lies inside $p_\sigma$? \\
$(b)$ Is there any singularities that some element in $H_{\sigma }$ lies outside the $p_{\sigma}$? \\

\noindent  For the first two questions, we don't have an answer yet but the answer for $2(b)$ is positive as the following example shows: Let $X$ be the hypersurface defined by  $f(x,y,z)=y^3+xz^2-x^4$ 
The dual Newton polyhedron $DNP(f)$ consists of three $3$-dimensional cones; these cones and their Hilbert bases are:

\begin{align*}
	\sigma_1 & =<e_1,e_3, u_1, u_2> & H_{\sigma_1} & =\{e_1,e_3,u_1,u_2,(1,1,1),(3,4,5)\}\\ 
	\sigma_2 & =<e_2,u_1,u_2>& H_{\sigma_2} & =\{e_2,u_1,u_2,(1,1,0),(2,1,0),(1,1,1),(2,3,3)\}\\
	\sigma_3 & =<e_2,e_3,u_2>& H_{\sigma_3} & =\{e_2,e_3,u_2,(1,2,2),(2,3,3),(3,4,5)\}
\end{align*}
\noindent where $u_1=(3,1,0), u_2=(6,8,9)$. The profile $p_{\sigma_3}$ of $\sigma_3$ is defined by the hyperplane $H: 8x-3y-3z+3=0$. But, the following figure shows that the element $(1,2,2)$ from $H_{\sigma_3}$ is outside of $p_{\sigma_3}$. The set $H_{DNP(f)}$ still give a minimal toric embedded resolution of the singularity.

\begin{figure}[h]
	\centering
	\includegraphics[width=0.7\textwidth]{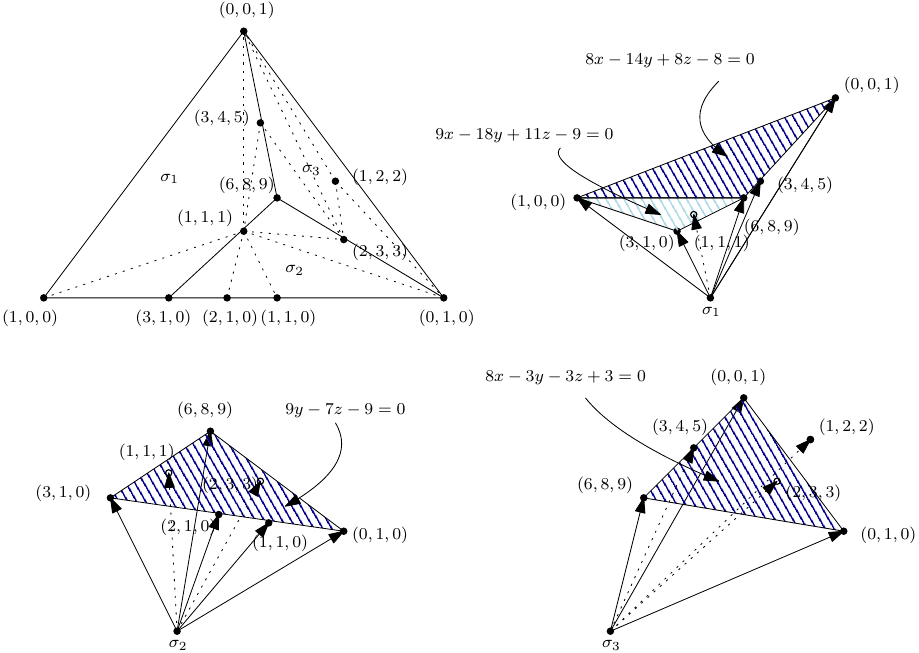}
	\caption{Profiles of the example}
\end{figure}

\newpage
\noindent Note that the hypersurface in this example has elliptic singularities and it is Newton non-degenerate. It is then natural to ask if it is a characterization of rational singularities, or just a question of choice of coordinates.

\section{Remarks on the Gr\"{o}bner fan of $X$}

\noindent Let $X$ be defined by $f(x,y,z)=0$. Let $\textbf{w}=(w_1,w_2,w_3)\in \mathbb{R}^3_{>0}$. The number
$$o_\textbf{w}(f):=min \{w_1a_1+w_2a_2+w_3a_3 \mid (a_1,a_2,a_3)\in S(f)\}$$
is called the $\textbf{w}$-order of $f$. The polynomial  
$$In_\textbf{w}(f):=\sum_{\{(a_1,a_2,a_3)\in S(f)\mid w_1a_1+w_2a_2+w_3a_3=o_w(f)\}} c_{(a_1,a_2,a_3)}x^{a_1}y^{a_2}z^{a_3}$$	
is called the $\textbf{w}$-initial form of $f$. We say that $\textbf{u}$ is equivalent to $\textbf{w}$ if $In_{\textbf{u}}(f)=In_{\textbf{w}}(f)$. The closure of the set
$$C_\textbf{w}(f):=\{ \textbf{u}\in {\mathbb R^3}\mid In_\textbf{w}(f)=In_{\textbf{u}}(f)\}$$
is a cone, called Gr\"{o}bner cone of $f$. The union of Gr\"{o}bner cones of $f$ form a fan, called the Gr\"{o}bner fan of $X$, denoted by $\mathcal{G}(X)$ (see \cite{Jensen} for more detailes), which  is introduced by T. Mora and L. Robbiano in \cite{Mora}. The full-dimensional cones in $\mathcal{G}(X)$ are in correspondence with the distinct monomials in $f$ \cite{Fukuda}. The set $$\mathcal{T}(f):=\{\textbf{u}=\in \mathbb{R}^3 \mid In_{\textbf{u}} (f) \ \text{is not a monomial} \}$$
\noindent is called the tropical variety of $f$.


\begin{prop}
	The tropical variety of an RTP-singularity is exactly the minimal abstract resolution of the singularity.
\end{prop}

\begin{proof} As before we provide the details for $B_{k,n}$-singularities:\\
	
	\noindent For $B_{2r-1,n}$-singularities, we look for all the vectors $w_i\in \mathbb{N}^3$, $1\leq i \leq 4$ for which $In_{w_1}(f)=f$, $In_{w_2}(f)=x^{2n+3}z-x^ry^2$, $In_{w_3}(f)=x^{2n+3}z-y^2z$ and $In_{w_4}(f)=-x^ry^2-y^2z$. This gives the following Gr\"{o}bner cones in $\mathcal{G}(X)$:
	
	${\bar{C}_{w_1}}=<(2,2n+3,2r)>$,
	
	${\bar{C}_{w_2}}=<(0,1,2),(2,2n+3,2r)>$,		
	
	${\bar{C}_{w_3}}=<(2,2n+3,0),(2,2n+3,2r)>$,		
	
	${\bar{C}_{w_4}}=<(1,0,r),(2,2n+3,2r)>$.\\
	
	\noindent	For $B_{2r,n}$-singularities, we look for all the vectors $w_i\in \mathbb{N}^3$, $1\leq i \leq 4$ for which  $In_{w_1}(f)=f$, $In_{w_2}(f)=x^{n+r+2}y-x^{2n+3}z$, $In_{w_3}(f)=-x^{2n+3}z+y^2z$and $In_{w_4}(f)=x^{n+r+2}y+y^2z$. This gives the following Gr\"{o}bner cones of $\mathcal{G}(X)$:
	
	${\bar{C}_{w_1}}=<(2,2n+3,2r+1)>$,
	
	${\bar{C}_{w_2}}=<(0,1,1),(2,2n+3,2r+1)>$,
	
	${\bar{C}_{w_3}}=<(2,2n+3,0), (2,2n+3,2r+1)>$,
	
	${\bar{C}_{w_4}}=<(1,0,n+r+2), (2,2n+3,2r+1)>$. \\

	\noindent In both cases, comparing with Figure $1$ above, the union $\bar{C}_{w_1}\cup \bar{C}_{w_2} \cup \bar{C}_{w_3} \cup \bar{C}_{w_4}$ is the abstract resolution of $B_{k,n}$-singularities.
	
\end{proof}

\begin{rem}
	Let $f$ defines an RDP-singularity. Let $\mathcal{J}(f)$ be the set of vector appearing in the jet graph of $f$. The intersection $\mathcal{G}(X)\cap \mathcal{J}(f)$ is exactly the Hilbert basis of $DNP(f)$, so gives the minimal toric embedded resolution of the singularity. This is not always true for RTP-singularities. For example, in the case of $E_{60}$-singularity, the vector $\textbf{w}=(2,3,3)$ for which $In_{w}(f)=z^3$ is in the intersection but it is not in Hilbert basis of $DNP(f)$. It is important to notice that this vector is not revealed in building the toric embedded resolution of the singularity. Hence $\mathcal{G}(X)\cap \mathcal{J}(f)$ also gives a toric embedded resolution of an RTP-singularity, which may not be minimal.
\end{rem}


{\scriptsize
	
	\renewcommand*{\arraystretch}{0,5}
	\begin{table}[htbp]
		\begin{sideways}
			\begin{tabular}{ | p{5cm}  | p{6cm} | p{11cm} | }
				\hline
				{\bf Type of  $f$ } &  {\bf Subprofiles} &  {\bf EV $\sim$ Hilbert Basis} \\
				
				\hline 
				\vskip.9cm $A_{k,l,m}: y^{3m+3}+xy^{m+1}z+xz^2-z^3=0$
				
				\ \
				
				$k=l=m>1$ &
				\vskip.2cm
				
				\ \

				$\includegraphics[width=6cm]{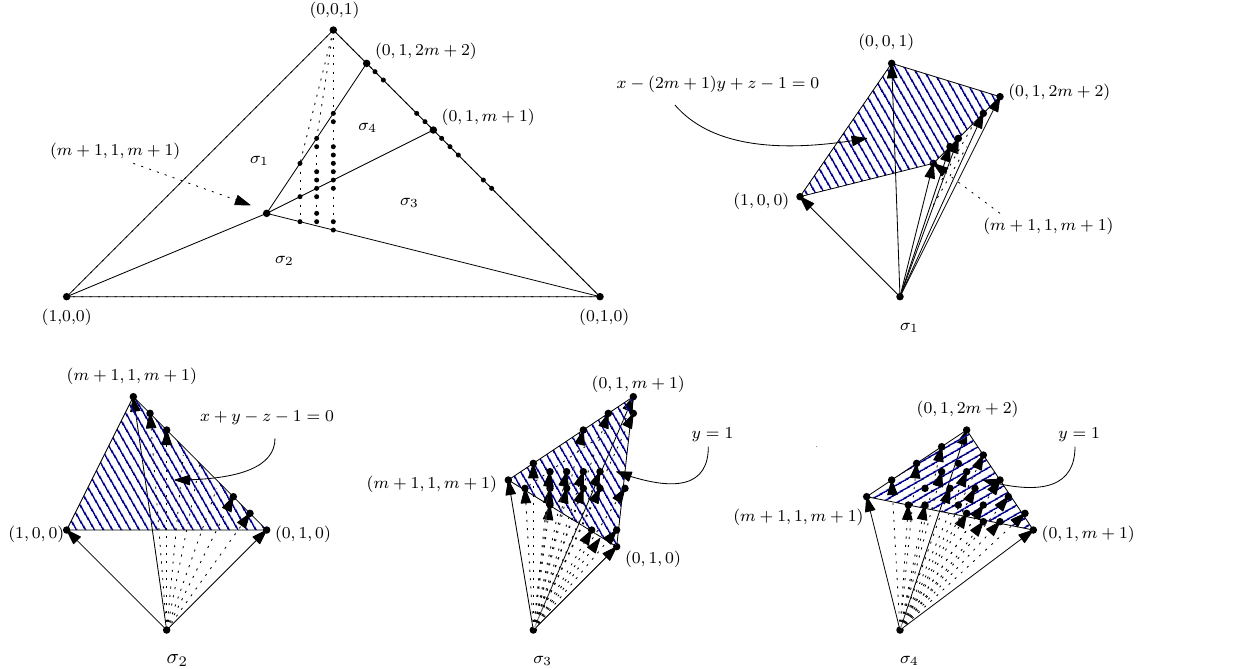}$ \caption {}&
				
				\ \
				
				\ \
				
				$H_{\sigma_1}=\{(1,0,0),(0,0,1),(0,1,2m+2),(m+1,1,m+1),(1,1,2m+1),(2,1,2m),$
				
				$(3,1,2m-1),\ldots,(m-1,1,m+3),(m,1,m+2)\} $
				
				$H_{\sigma_2}=\{(1,0,0),(0,1,0),(m+1,1,m+1),(1,1,1),(2,1,2),\ldots,(m-1,1,m-1),(m,1,m)\}$
				
				$H_{\sigma_3}=\{(0,1,0),(0,1,1),(1,1,1),(0,1,2),(1,1,2),(2,1,2),\ldots,(0,1,m),(1,1,,m),\ldots,$
				
				$(m,1,m),(0,1,m+1),(1,1,m+1),\ldots,(m+1,1,m+1)\}$
				
				$H_{\sigma_4}=\{(0,1,m+1),(0,1,m+2),\ldots,(0,1,2m+2),(1,1,m+1),(1,1,m+2),\ldots,$
				
				$(1,1,2m+1),\ldots,(m-2,1,m+1),(m-2,1,m+1),(m-2,1,m+2),\ldots,$
				
				$(m-2,1,m+4),(m-1,1,m+1),(m-1,1,m+2),(m-1,1,m+3),(m,1,m+1),$
				
				$(m,1,m+2),(m,1,m)\}$
				
				\\
				\hline
				
				\vskip.9cm $A_{k,l,m}:y^{k+l+m+3}+y^{2k+2}z+y^{k+1}z^2+xy^{k+1}z+xz^2-z^3=0 $  
				
				\ \
				
				$k=l<m$,
				
				\ \
				
				$k,l,m\geq 1$&
				\vskip.4cm
				
				\ \

				$\includegraphics[width=6cm]{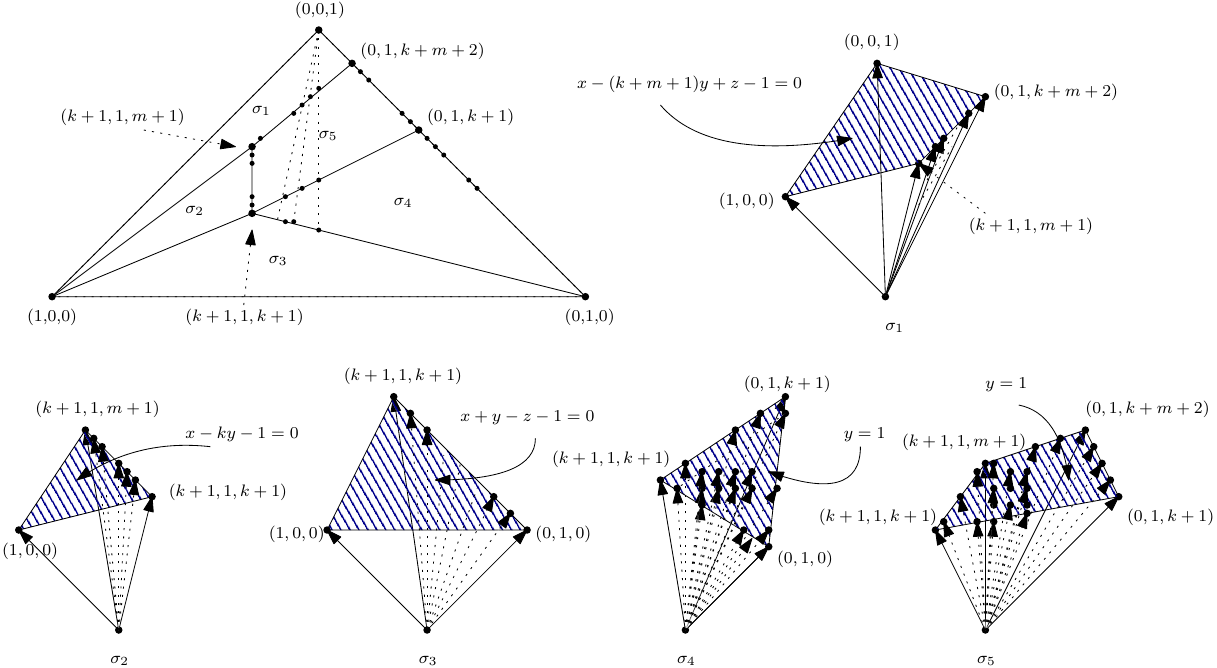}$  \caption{} & 
				
				\ \
				
				\ \
				
				$H_{\sigma_1}=\{(1,0,0),(0,0,1),(0,1,k+m+2),(1,1,k+m+1),(2,1,k+m),\dots,(k+1,1,m+1)\}$
				
				$H_{\sigma_2}=\{(1,0,0),(k+1,1,k+1),(k+1,1,k+2),\ldots,(k+1,1,m),(k+1,1,m+1)\}$
				
				$H_{\sigma_3}=\{(1,0,0),(0,1,0),(1,1,1),(2,1,2),\ldots,(k,1,k),(k+1,1,k+1)\}$
				
				$H_{\sigma_4}=\{(0,1,1),(1,1,1),(0,1,2),(1,1,2),(2,1,2),(0,1,3),\ldots,(3,1,3),(0,1,4),\ldots,$
				
				$(4,1,4),\ldots,(0,1,k),(1,1,k),\ldots,(k,1,k),(0,1,k+1),(1,1,k+1),\ldots,(k+1,1,k+1)\}$
				
				$H_{\sigma_5}=\{(0,1,k+1),(1,1,k+1),\ldots,(k+1,1,k+1),(0,1,k+2),(1,1,k+2),\ldots,$
				
				$(k+1,1,k+2),\ldots(0,1,m+1),(1,1,m+1),\ldots,(k+1,1,m+1),(0,1,m+2),$
				
				$(1,1,m+2),\ldots,(k,1,m+2),\ldots,(0,1,k+m),(1,1,k+m),(2,1,k+m),(0,1,k+m+1),$
				
				$(1,1,k+m+1),(0,1,k+m+2)\}$
				
				\\
				
				\hline
				
				\vskip.9cm $A_{k,l,m}: y^{3k}+y^{2k+m+l-2}-2y^{l+k}z-xy^kz+y^{m}z^2+xz^2-z^3=0 $  
				
				\ \
				
				$l<m<k$
				
				\ \
				
				$l+k>2m$ 
				
				\ \
				
				$m+l-2>k$ 
				
				\ \
				
				$k,l,m\geq 1$ &
				\vskip.4cm
				
				\ \
				
				\ \
				
				$\includegraphics[width=6cm]{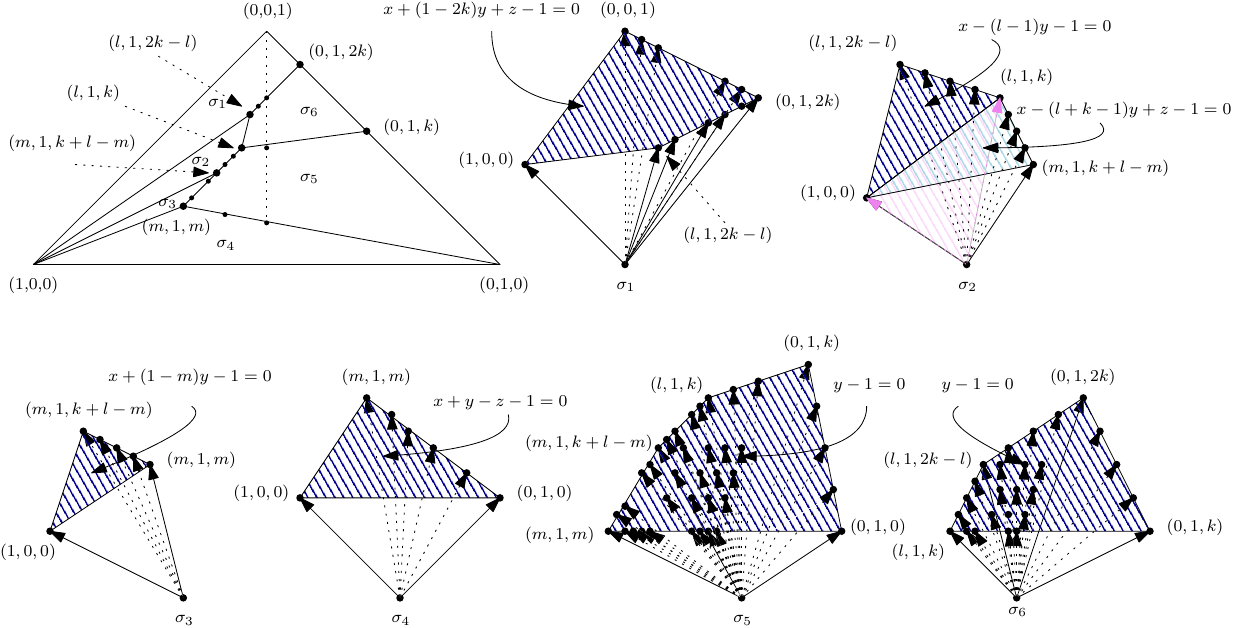}$  \caption{} &
				
				\ \
				
				\ \
				
				$H_{\sigma_1}=\{(1,0,0),(0,0,1),(l,1,2k-l),(0,1,2k),(0,1,1),(0,1,2),\ldots,(0,1,2k-1),$
				
				$(1,1,2k-1),(2,1,2k-2),\ldots,(l-1,1,2k-l+1)\}$
				
				$H_{\sigma_2}=\{(1,0,0),(m,1,k+l-m),(l,1,k),(l,1,2k-l),(m-1,1,k+l-m+1),\ldots,$
				
				$(l+1,1,k-1),(l,1,k+1),\ldots,(l,1,2k-l-1)\}$
				
				$H_{\sigma_3}=\{(1,0,0),(m,1,m),(m,1,k+l-m),(m,1,m+1),(m,1,m+2),\ldots,$
				
				$(m,1,k+l-m-1)\}$
				
				$H_{\sigma_4}=\{(1,0,0),(0,1,0),(m,1,m),(1,1,1),(2,1,2),\ldots,(m-1,1,m-1)\}$
				
				$H_{\sigma_5}=\{(0,1,0),(m,1,m),(m,1,k+l-m),(l,1,k),(0,1,k),(0,1,1),(0,1,2),\ldots,$
				
				$(0,1,k-1),(1,1,1),\ldots,(1,1,k),(2,1,2),\ldots,(2,1,k),\ldots,(l-1,1,l-1)\ldots,$
				
				$(l-1,1,k),(l,1,l),(l,1,l+1),\ldots,(l,1,k),(l+1,1,l+1),(l+1,1,l+2),\ldots,$
				
				$(l+1,1,k-1),\ldots,(m-2,1,m-2),\ldots(m-2,1,k+l-m+2),(m-1,1,m),\ldots,$
				
				$(m-1,1,k+l-m+1)\}$
				
				$H_{\sigma_6}=\{(0,1,k),(0,1,2k),(l,1,k),(l,1,2k-1),(0,1,k+1),(0,1,k+2),\ldots,$
				
				$(0,1,2k-1),(1,1,k),(1,1,k+1),\ldots,(1,1,2k-1),(2,1,k),(2,1,k+1),\ldots,$
				
				$(l-1,1,2k-l+1),(l,1,k+1),\ldots,(l,1,2k-l-1)\}$

				\\
				
				\hline
			\end{tabular}
		\end{sideways}
	\end{table}

	\renewcommand*{\arraystretch}{0,5}
	\begin{table}[htbp]
		\begin{sideways}
			\begin{tabular}{ | p{5cm}  | p{6cm} | p{11cm} | }
				\hline			
				{\bf Type of  $f$ } &  {\bf Subprofiles} &  {\bf EV $\sim$ Hilbert Basis} \\

				\hline
				\vskip.6cm $A_{k,l,m}: y^{3k}+y^{2k+m+l-2}-2y^{l+k}z-xy^kz+y^{m}z^2+xz^2-z^3=0 $  
				
				\ \
				
				$l<m<k$ 
				
				\ \
				
				$l+k>2m$ 
				
				\ \
				
				$m+l-2\leq k$
				
				\ \
				
				$k,l,m\geq 1$  &
				\vskip.4cm
				
				\ \

				$\includegraphics[width=6cm]{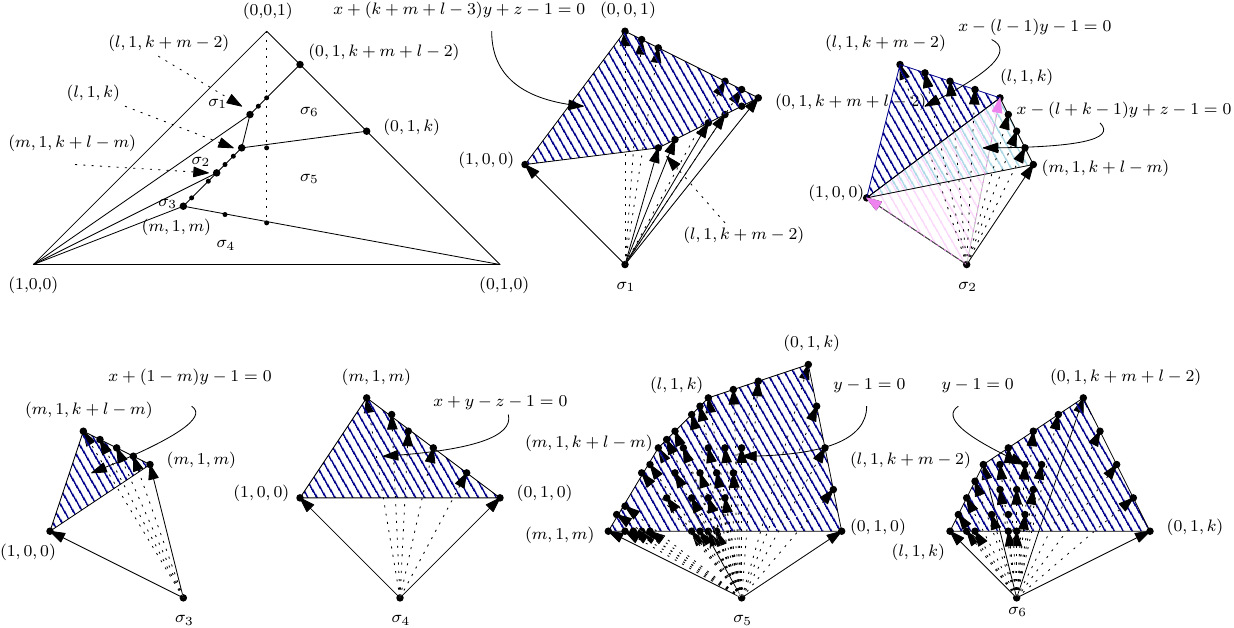}$   \caption{} & 
				
				\ \
				
				\ \
				
				$H_{\sigma_1}=\{(1,0,0),(0,0,1),(l,1,k+m-2),(0,1,k+m+l-2),(1,1,k+m+l-3),\ldots,$
				
				$(l-1,1,k+m-1),(0,1,1),\ldots,(0,1,2k-1)\}$
				
				$H_{\sigma_2}=\{(1,0,0),(m,1,k+l-m),(l,1,k),(l,1,k+m-2),(m-1,1,k+l-m+1),\ldots,$
				
				$(l+1,1,k-1),(l,1,k+1),\ldots,(l,1,k+m-3)\}$
				
				$H_{\sigma_3}=\{(1,0,0),(m,1,m),(m,1,k+l-m),(m,1,m+1),(m,1,m+2),\ldots,$
				
				$(m,1,k+l-m-1)\}$
				
				$H_{\sigma_4}=\{(1,0,0),(0,1,0),(m,1,m),(1,1,1),(2,1,2),\ldots,(m-1,1,m-1)\}$
				
				$H_{\sigma_5}=\{(0,1,0),(m,1,m),(m,1,k+l-m),(l,1,k),(0,1,k),(0,1,1),\ldots,$
				
				$(0,1,k-1),(1,1,1),\ldots,(1,1,k),\ldots,(l,1,l),(l,1,l+1),\ldots,(l,1,k),(l+1,1,l+1),\ldots,$
				
				$(l+1,1,k-1),\ldots,(m-1,1,m-1),\ldots,(m-1,1,k+l-m-1)\}$
				
				$H_{\sigma_6}=\{(0,1,k),(0,1,k+m+l-2),(l,1,k),(l,1,k+m-2),(0,1,k+m),$
				
				$(0,1,k+2),\ldots,(0,1,k+m+l-3),(1,1,k),(1,1,k+1),\ldots,(1,1,k+m+l-3),\ldots$
				
				$(l-1,1,k),(l-1,1,k+1),\ldots,(l-1,1,k+m-1),(l,1,k+1),(l,1,k+2),\ldots,(l,1,k+m-3)\}$
				
				\\
				\hline

				\vskip.5cm $A_{k,l,m}: y^{3k}+y^{2k+m+l-2}-2y^{l+k}z-xy^kz+y^{m}z^2+xz^2-z^3=0 $
				
				\ \
				
				$l<m<k$
				
				\ \
				
				$l+k\leq 2m$ 
				
				\ \
				
				$m+l-2>k$, $l+k$ is even 
				
				\ \
				
				$k,l,m\geq 1$  &
				\vskip.4cm
				
				\ \

				$\includegraphics[width=6cm]{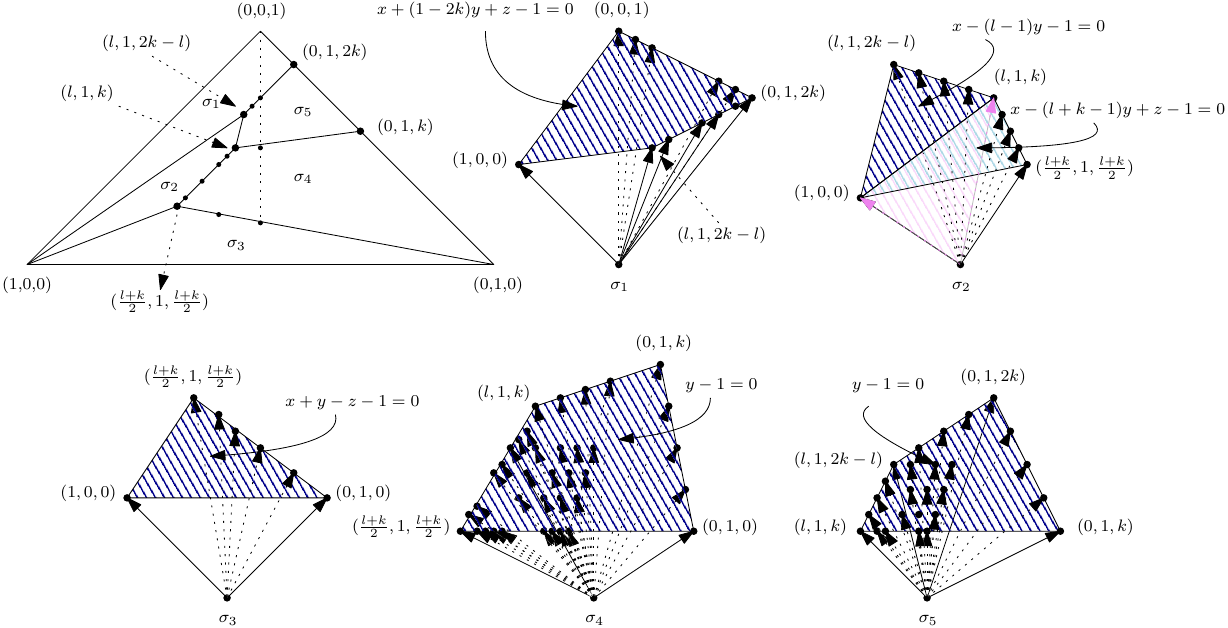}$   \caption{} &
				
				\ \
				
				\ \
				
				$H_{\sigma_1}=\{(1,0,0),(0,0,1),(0,1,2k),(l,1,2k-l),(0,1,1),(0,1,2),\ldots(0,1,2k-1),$
				
				$(1,1,2k-1),(1,1,2k-2),(1,1,2k-l+1)\} $
				
				$H_{\sigma_2}=\{(1,0,0),(\frac{l+k}{2},1,\frac{l+k}{2}),(l,1,k),(l,1,2k-l),(\frac{l+k}{2},1,\frac{l+k}{2}+1),(\frac{l+k}{2}-2,1,\frac{l+k}{2}+2),$
				
				$\ldots,(l+1,1,k-1),(l,1,k+1),\ldots,(l,1,2k-l-1)\}$
				
				$H_{\sigma_3}=\{(1,0,0),(0,1,0),(\frac{l+k}{2},1,\frac{l+k}{2}),(1,1,1),(2,1,2),\ldots,(\frac{l+k}{2}-1,1,\frac{l+k}{2}-1)\}$
				
				$H_{\sigma_4}=\{(0,1,0),(0,1,k),(1,1,k),(\frac{l+k}{2},1,\frac{l+k}{2}),(0,1,1),(0,1,2),\ldots,(0,1,k-1),(1,1,1),$
				
				$(1,1,2),\ldots,(1,1,k),\ldots,(l-1,1,l-1),(l-1,1,l),\ldots,(l-1,1,k),(l,1,l),(l,1,l+1),$
				
				$\ldots,(l,1,k-1),(l+1,1,l+1),(l+1,1,l+2),\ldots,(l+1,1,k-1),(l+2,1,l+2),(l+2,1,l+3),$
				
				\noindent$\ldots,(l+2,1,k-2),\ldots,(\frac{l+k}{2},1,\frac{l+k}{2}-1),\ldots,(\frac{l+k}{2}-1,1,\frac{l+k}{2}+1)\}$
				
				$H_{\sigma_6}=\{(0,1,k),(0,1,2k),(l,1,k),(l,1,2k-l),(0,1,k+1),(0,1,k+2),\ldots(0,1,2k-1),$
				
				$(1,1,k),\ldots,(1,1,2k-1),\ldots,(l-1,1,k),(l-1,1,k+1),\ldots,(l-1,1,2k-l+1),$
				
				$(l,1,k+1),(l,1,k+2),\ldots,(l,1,2k-l-1)\}$
				
				\\

				\hline
				\vskip.5cm $A_{k,l,m}: y^{3k}+y^{2k+m+l-2}-2y^{l+k}z-xy^kz+y^{m}z^2+xz^2-z^3=0 $  
				
				\ \
				
				$l<m<k$
				
				\ \
				
				$l+k\leq 2m$ 
				
				\ \
				
				$m+l-2\leq k$, $l+k$ is even 
				
				\ \
				
				$k,l,m\geq 1$   &
				\vskip.4cm
				
				\ \

				$\includegraphics[width=6cm]{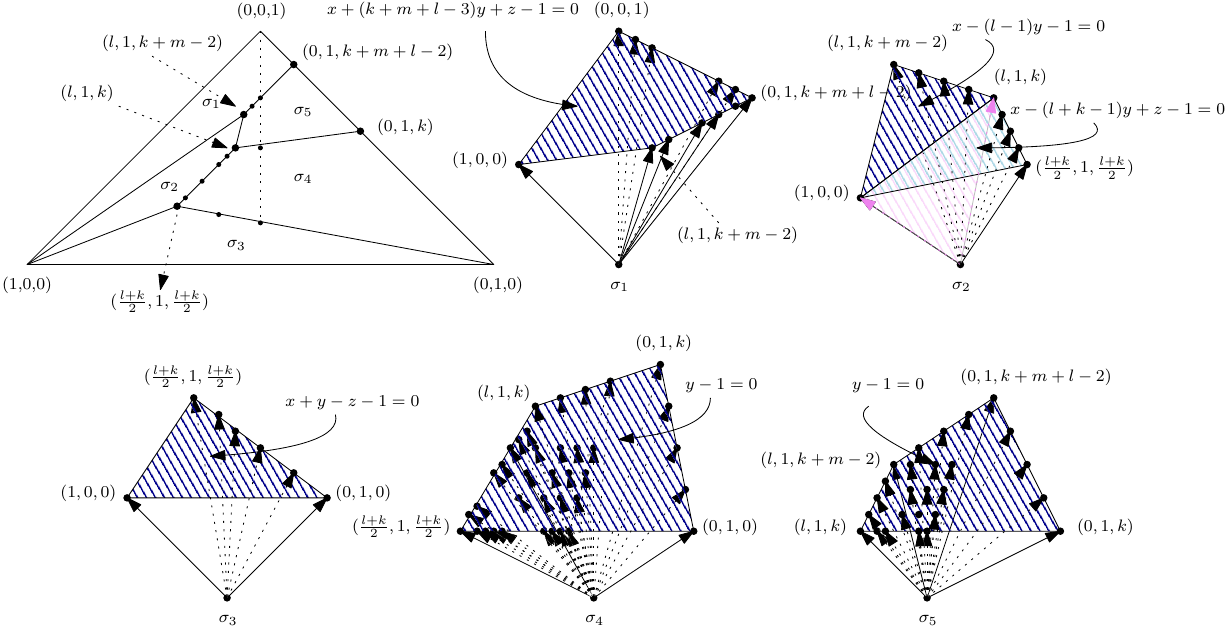}$   \caption{} &
				
				\ \
				
				\ \
				
				$H_{\sigma_1}=\{(1,0,0),(0,0,1),(0,1,k+m+l-2),(l,1,k+m-2),(0,1,1),(0,1,2),\ldots,$
				
				$(0,1,k+m+l-3),(1,1,k+m+l-3),(2,1,k+m+l-4),\ldots,(l-1,1,k+m-1)\} $
				
				$H_{\sigma_2}=\{(1,0,0),(\frac{l+k}{2},1,\frac{l+k}{2}),(l,1,k),(l,1,k+m-2),(\frac{l+k}{2}-1,1,\frac{l+k}{2}+1),$
				
				\noindent$(\frac{l+k}{2}-2,1,\frac{l+k}{2}+2),\ldots,(l+1,1,k-1),(l,1,k+1),\ldots,(l,1,k+m-3)\}$
				
				$H_{\sigma_3}=\{(1,0,0),(0,1,0),(\frac{l+k}{2},1,\frac{l+k}{2}),(1,1,1),\ldots,(\frac{l+k}{2}-1,1,\frac{l+k}{2}-1)\}$
				
				$H_{\sigma_4}=\{(0,1,0),(0,1,k),(1,1,k),(\frac{l+k}{2},1,\frac{l+k}{2}),(0,1,1),(0,1,2),\ldots,(0,1,k-1),(1,1,1),$
				
				$(1,1,2),\ldots,(1,1,k),\ldots,(l-1,1,l-1),(l-1,1,l),\ldots,(l-1,1,k),(l,1,l),(l,1,l+1),$
				
				$\ldots,(l,1,k-1),(l+1,1,l+1),(l+1,1,l+2),\ldots,(l+1,1,k-1),(l+2,1,l+2),(l+2,1,l+3),$
				
				\noindent$\ldots,(l+2,1,k-2),\ldots,(\frac{l+k}{2}-1,1,\frac{l+k}{2}-1),\ldots,(\frac{l+k}{2}-1,1,\frac{l+k}{2}+1)\}$
				
				$H_{\sigma_5}=\{(0,1,k),(0,1,k+m+l-2),(l,1,k),(l,1,k+m-2),(0,1,k+1),(0,1,k+2),$
				
				$\ldots(0,1,k+m+l-3),(1,1,k),(1,1,k+1),\ldots,(1,1,k+m+l-3),\ldots,(l-1,1,k),$
				
				$(l-1,1,k+1),\ldots,(l-1,1,k+m-1),(l,1,k+1),(l,1,k+2),\ldots,(l,1,k+m-3)\}$
				
				\ \			
				\\
				
				\hline
				
			\end{tabular}
		\end{sideways}
	\end{table}

	\renewcommand*{\arraystretch}{0,5}
	\begin{table}[htbp]
		\begin{sideways}
			\begin{tabular}{ | p{5cm}  | p{6cm} | p{11cm} | }
				\hline			
				{\bf Type of  $f$ } &  {\bf Subprofiles} &  {\bf EV $\sim$ Hilbert Basis} \\
				\hline
				\vskip.5cm $A_{k,l,m}: y^{2k+m}+y^{m+k}z-y^{l+k}z+xy^kz+-y^kz^2+y^lz^2+xz^2-z^3=0 $  
				
				\ \
				
				$l<m<k$
				
				\ \
				
				$l+k\leq 2m$ 
				
				\ \
				
				$l+k$ is odd 
				
				\ \
				
				$k,l,m\geq 1$ &
				
				\vskip.4cm
				
				\ \
				
				$\includegraphics[width=6cm]{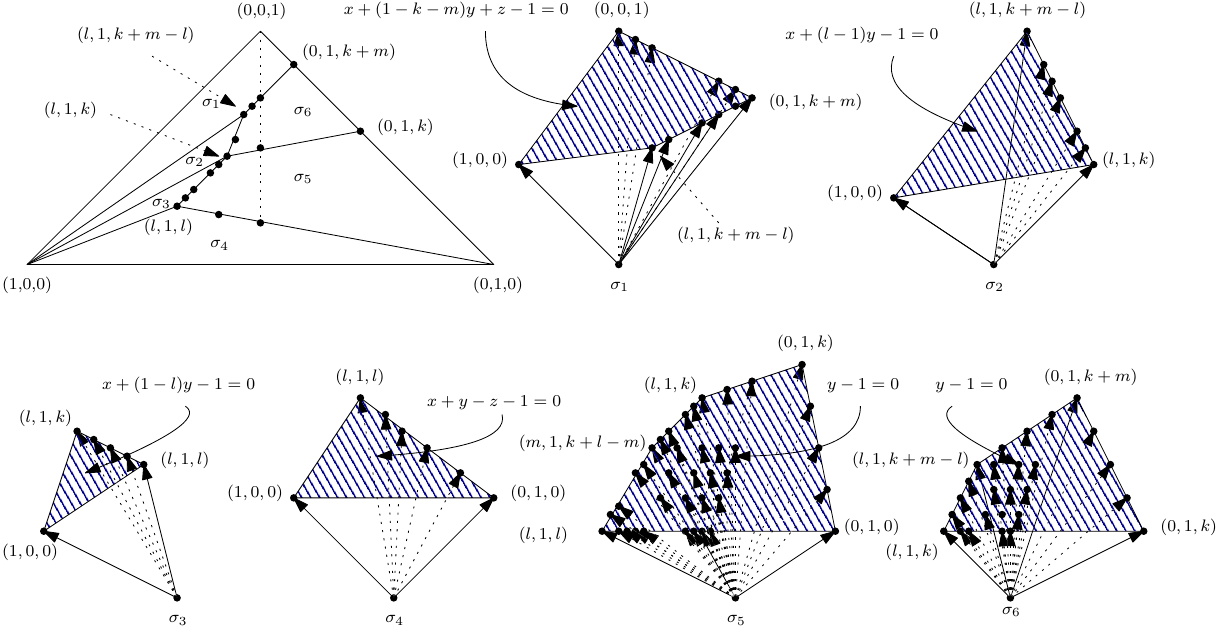}$  \caption{}  &
				
				\ \
				
				\ \
				
				$H_{\sigma_1}=\{(1,0,0),(0,0,1),(0,1,k+m),(l,1,k+m-l),(1,1,k+m-1),(2,1,k+m-2),$
				
				$\ldots,(l-1,1,k+m-l-1),(0,1,1),(0,1,2),\ldots,(0,1,k+m+1)\}$
				
				$H_{\sigma_2}=\{(1,0,0),(l,1,k),(l,1,k+m-l-1),(l,1,k+1),(l,1,k+2),\ldots,(l,1,k+m-l-1)\}$
				
				$H_{\sigma_3}=\{(1,0,0),(l,1,l),(l,1,k-1),(l,1,l+1),(l,1,l+2),\ldots,(l,1,k-1)\}$
				
				$H_{\sigma_4}=\{(1,0,0),(0,1,0),(l,1,l),(1,1,1),(2,1,2),\ldots,(l-1,1,l-1)\}$
				
				$H_{\sigma_5}=\{(0,1,0),(0,1,1),\ldots,(0,1,k),(1,1,1),(1,1,2),\ldots,(1,1,k),(2,1,2),\ldots,(2,1,k),$
				
				$\ldots,(l-1,1,l-1),(l-1,1,l),\ldots,(l-1,1,k),(l,1,l),(l,1,l+1),\ldots,(l,1,k)\}$
				
				$H_{\sigma_6}=\{(0,1,k),(0,1,k+1),\ldots,(0,1,k+m),(1,1,k),(1,1,k+1),\ldots,(1,1,k+m-1),$
				
				$(2,1,k),(2,1,k+1),\ldots,(2,1,k+m-2),\ldots,(l-1,1,k),(l-1,1,k+1),\ldots,$
				
				$(l-1,1,k+m-l-1),(l,1,k),(l,1,k+1),\ldots,(l,1,k+m-l)\}$

				\\
				
				\hline
				\vskip.9cm $C_{n,m}: x^{n-1}y^{2m+2}+y^{2m+4}-xz^2=0 $ \ \
				
				\ \
				
				$n$ is odd, $n\geq 3$, $m\geq 2$   &
				\vskip.4cm
				
				\ \
				
				$\includegraphics[width=5.4cm]{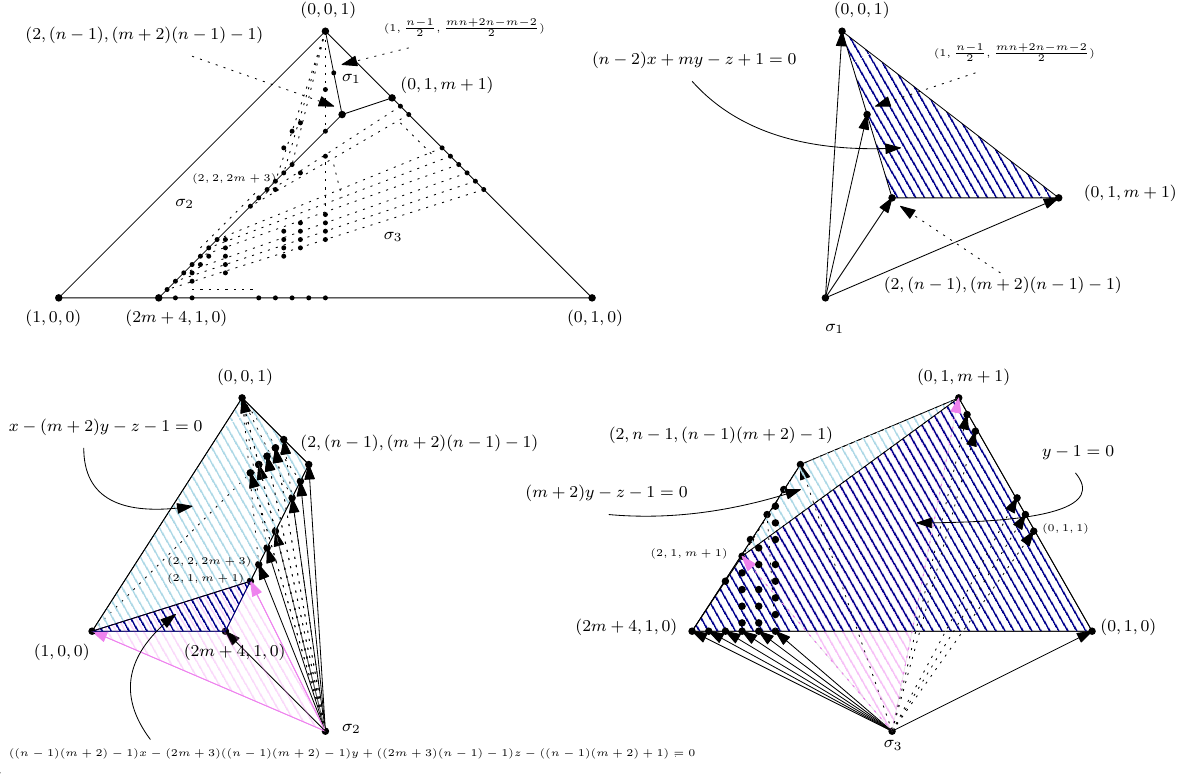}$   \caption{} &
				
				\ \
				
				\ \
				
				$H_{\sigma_1}=\{(0,0,1),(0,1,m+1),(2,n-1,mn+2n-m-3),(1,\frac{n-1}{2}, \frac{mn+2n-m-2}{2})\}$
				
				$H_{\sigma_2}=\{(1,0,0),(0,0,1),(2m+4,1,0), (2,n-1,mn+2n-m-3),(2m+2,1,1), (2m,1,2),$
				
				$ (2m-2,1,3), \ldots, (4,1,m), (2,1,m+1), (2,2,2m+3),(2,3,3m+5), \ldots,$
				
				$ (2,n-2,mn+2n-2m-5), (1,1,m+2),(1,2,2m+4),(1,3,3m+6),\ldots, $
				
				$(1, \frac{n-1}{2}, \frac{mn+2n-m-2}{2})\}$
				
				$H_{\sigma_3}=\{(0,1,0), (2m+4,1,0), (0,1,m+1),(2,n-1,mn+2n-m-3),(1,1,0), (2,1,0),$
				
				$ \ldots, (2m+3,1,0), (0,1,1), (1,1,1), \ldots,(2m+1,1,1), (2m+2,1,1), (0,1,2),(1,1,2),$
				
				$\ldots,(2m-1,1,2), (2m,1,2), \ldots,(0,1,m-1),(1,1,m-1),\ldots(5,1,m-1),(6,1,m-1),$
				
				$(0,1,m),(1,1,m),(2,1,m),(3,1,m),(4,1,m),(1,1,m+1),(1,2,2m+3),(1,3,3m+5),$
				
				$\ldots, (1,\frac{n-1}{2},\frac{mn+2n-m-4}{2}), (2,1,m+1),(2,2,2m+3),(2,3,3m+5),\ldots,$
				
				$(2,n-2,mn+2n-2m-5) \}$
				
				\\
				\hline 
				\vskip.9cm $C_{n,m}: x^{n-1}y^{2m+2}+y^{2m+4}-xz^2=0 $ \ \

				\ \
				
				$n$ is even, $n\geq 3$, $m\geq 2$   &
				\vskip.2cm
				
				\ \

				$\includegraphics[width=5.4cm]{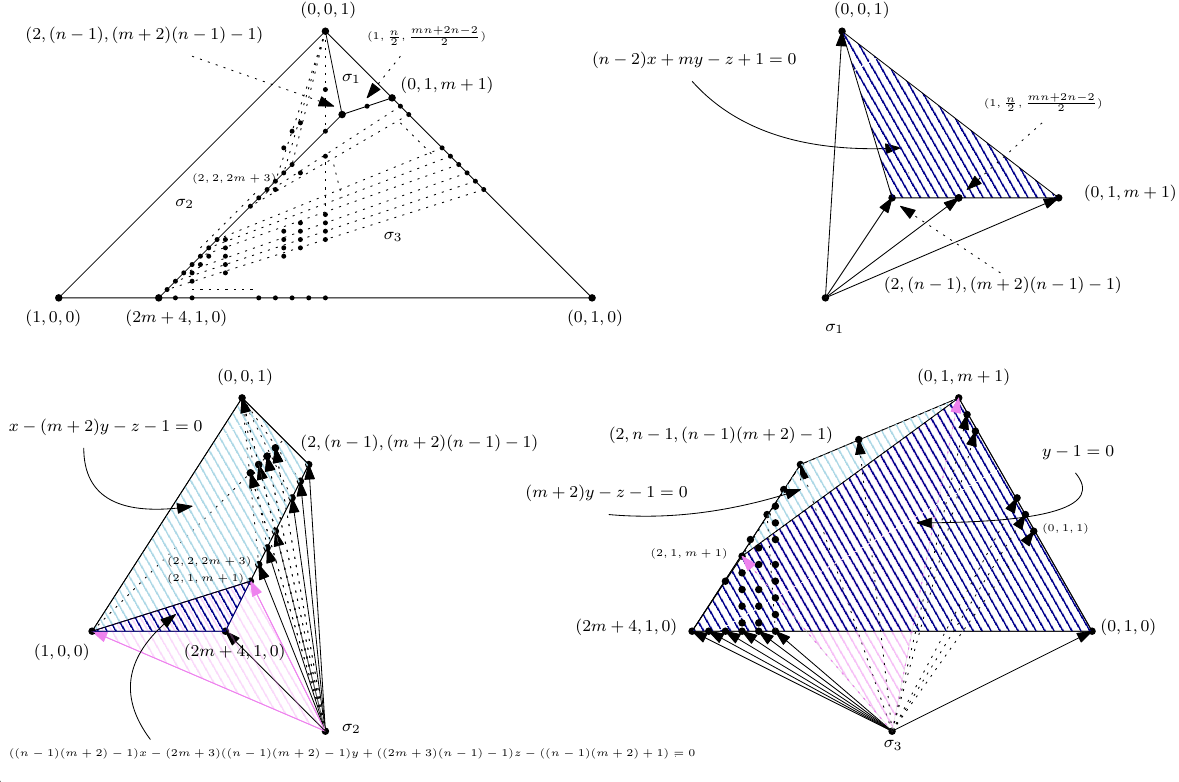}$   \caption{} &

				\ \
				
				\ \
				
				$H_{\sigma_1}=\{(0,0,1),(0,1,m+1),(2,n-1,mn+2n-m-3),(1,\frac{n}{2}, \frac{mn+2n-2}{2}) \}$
				
				$H_{\sigma_2}=\{(1,0,0),(0,0,1),(2m+4,1,0),(2,n-1,mn+2n-m-3),(2m+2,1,1),(2m,1,2),$
				
				$\ldots, (2,1,m+1),(2,2,2m+3),(2,3,3m+5),\ldots, (2,n-2,mn+2n-2m-5),(1,1,m+2),$
				
				$(1,2,2m+4),\ldots,(1,\frac{n-2}{2},\frac{mn+2n-2m-4}{2})\}$
				
				$H_{\sigma_3}=\{(0,1,0), (2m+4,1,0), (0,1,m+1),(2,n-1,mn+2n-m-3),(1,1,0), (2,1,0),$
				
				$ \ldots, (2m+3,1,0), (0,1,1), (1,1,1), \ldots,(2m+1,1,1), (2m+2,1,1), (0,1,2),(1,1,2),$
				
				$\ldots,(2m-1,1,2), (2m,1,2), \ldots,(0,1,m-1),(1,1,m-1),\ldots(5,1,m-1),(6,1,m-1),$
				
				$(0,1,m),(1,1,m),(2,1,m),(3,1,m),(4,1,m),(1,1,m+1),(1,2,2m+3),(1,3,3m+5),$
				
				$\ldots, (1,\frac{n}{2},\frac{mn+2n-2}{2}), (2,1,m+1),(2,2,2m+3),(2,3,3m+5),\ldots,$
				
				$(2,n-2,mn+2n-2m-5) \}$
				
				\\
				\hline
				
			\end{tabular}
		\end{sideways}
	\end{table}

	\renewcommand*{\arraystretch}{0,5}
	\begin{table}[htbp]
		\begin{sideways}
			\begin{tabular}{ | p{5cm}  | p{6cm} | p{11cm} | }
				\hline
				{\bf Type of  $f$ } &  {\bf Subprofiles} &  {\bf EV $\sim$ Hilbert Basis} \\
				\hline
				\vskip.9cm $D_{n-1}:z^2-xy^2-x^4=0 $ 
				
				\ \
				
				$n\geq 1$  &
				\vskip.4cm

				$\includegraphics[width=5.1cm]{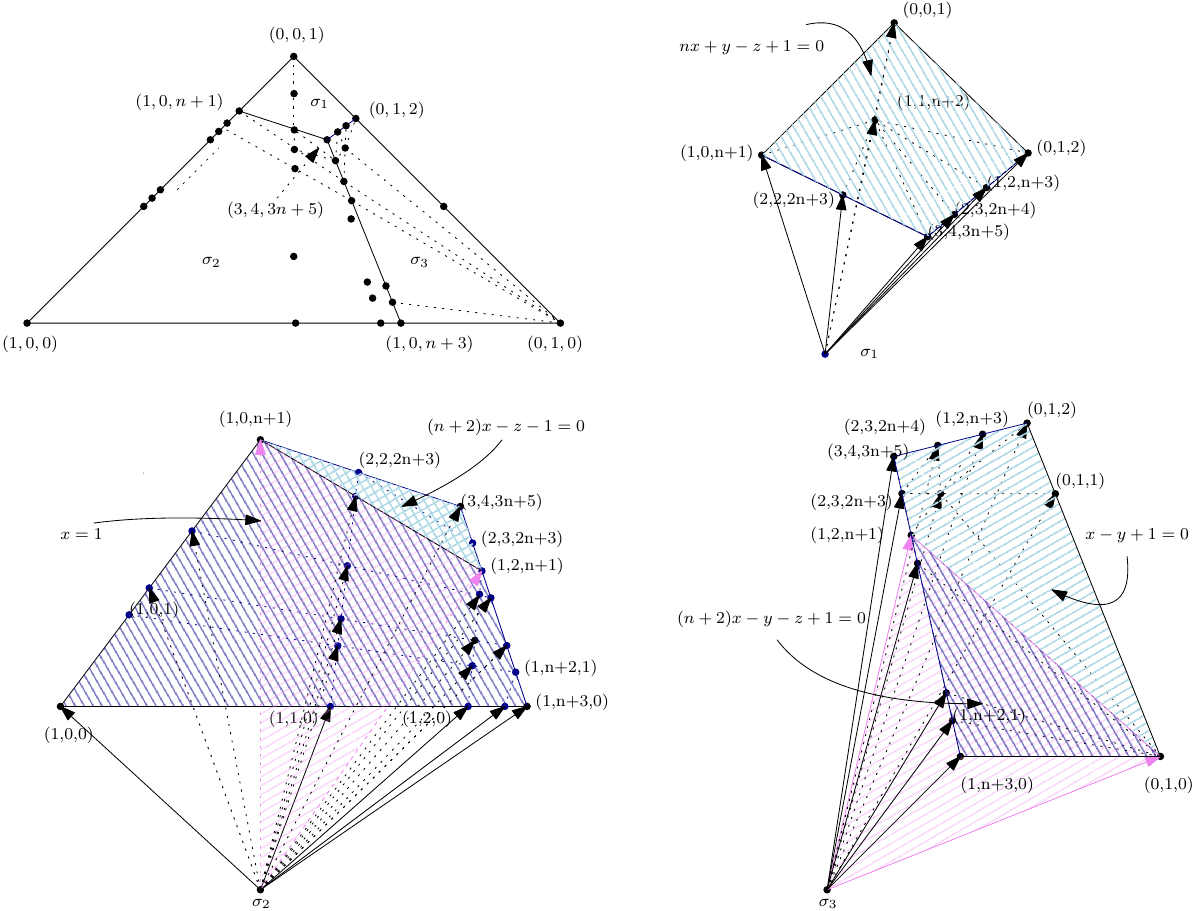}$  \caption{}  &
				
				\ \
				
				\ \
				
				$H_{\sigma_1}=\{(1,0,0),(1,0,n+1),(1,n+3,0),(3,4,3n+5),(1,0,1),(1,0,2),\ldots,(1,0,n),$
				
				$(1,1,0), (1,2,0),\ldots,(1,n+2,0),(1,1,1),(1,1,2),(1,1,n+1),(1,2,1),(1,2,2),(1,2,n),$
				
				$(1,3,1), (1,3,2),(1,3,n-1),(1,4,1),(1,4,2),(1,4,n-2),(1,n,1),(1,n,2),$ $(1,n+1,1),$
				
				$(2,2,2n+3),(2,3,2n+3),(1,2,n+1),(1,3,n),\ldots,(1,n+1,2),(1,n+2,1)\}$

				$H_{\sigma_2}=\{e_3,(0,1,2),(1,0,n+1),(3,4,3n+5),(1,2,n+3),(2,3,2n+4),(2,2,2n+3),$
				
				$(1,1,n+2)\}$ 
				$\{e_2,(0,1,2),(1,n+3,0),(3,4,3n+5),(0,1,1),(1,2,n+3),(2,3,2n+4),$
				
				$(2,3,2n+3),(1,2,n+2),(1,n+2,1),(1,n+1,2),(1,n,3),\ldots,(1,3,n),(1,2,n+1) \}$
				
				$H_{\sigma_3}=\{e_2,(0,1,2),(1,n+3,0),(3,4,3n+5),(0,1,1),(1,2,n+3),(2,3,2n+4),$
				
				$(2,3,2n+3),(1,2,n+2),(1,n+2,1),(1,n+1,2),(1,n,3),\ldots,(1,3,n),(1,2,n+1) \}$
				
				\\
				\hline 
				\vskip.9cm $F_{k-1}: y^{2k+3} + x^2y^{2k} – xz^2=0$ 
				
				\ \
				
				$k\geq 2$ &
				\vskip.2cm
				
				\ \
				
				$\includegraphics[width=5.1cm]{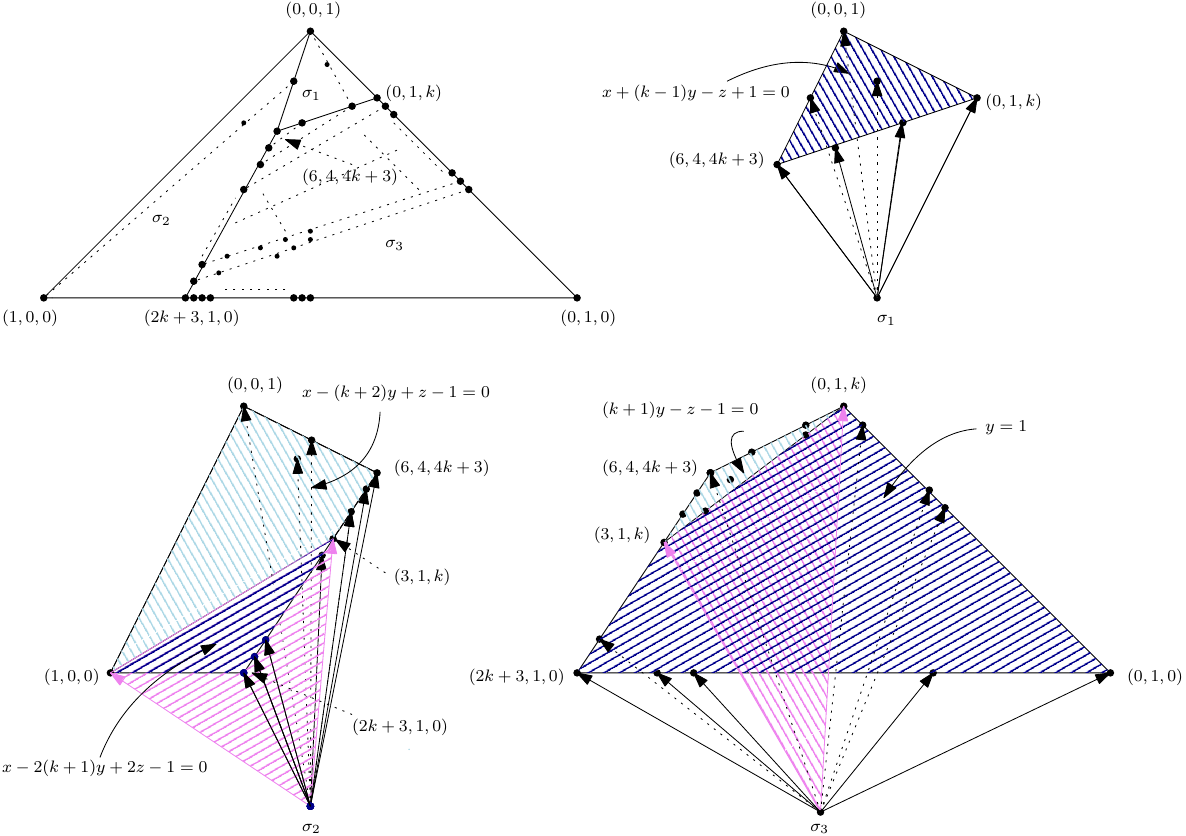}$   \caption{} & 
				
				\ \
				
				\ \
				
				$ H_{\sigma_1}=\{e_1,(0,1,k),(6,4,4k+3)
				(3,2,2k+2),(2k,2,2k+1),(4,3,3k+2),(1,1,k+1)\}$

				$H_{\sigma_2}=\{e_1,e_3,(2k+3,1,0),(6,4,4k+3)
				(2k+1,1,1),(2k-1,1,2),(2k-3,1,3),\ldots,$
				
				$(7,1,k-2),(5,1,k-1),(3,1,k),(4,2,2k+1),(5,3,3k+2),(3,2,2k+1), (2,1,k+1) \}$
				
				$H_{\sigma_3}=\{e_2,(0,1,k),(2k+1,3,0),(6,4,4k+3),(0,1,1),(0,1,2),\ldots,(0,1,k-1),(1,1,0),$
				
				$(2,1,0),\ldots,(2k+2,1,0),(2,2,2k+1),(4,3,3k+2),
				(5,3,3k+2),(4,2,2k+1),(3,1,k),$
				
				$(5,1,k-1),(2k+1,1,1),(2k-1,1,2),(2k-3,1,3),(2k-5,1,4),\ldots,(7,1,k-2),$
				
				$(1,1,1),(2,1,1),\ldots,(2k,1,1),
				(1,1,2),(2,1,2),\ldots,(2k-2,1,2),
				(1,1,3),(2,1,3),\ldots,$
				
				$(2k-4,1,3),\ldots(1,1,k-2),(2,1,k-2),\ldots,(6,1,k-2),(1,1,k-1),(2,1,k-1),\ldots,$
				
				$(4,1,k-1),(1,1,k),(2,1,k), (3,2,2k+1) \}$
				
				\\
				\hline
				
				\vskip.9cm $E_{6,0}:z^3+y^3z+x^2y^2=0 $  
				
				&
				\vskip.4cm
				$\includegraphics[width=5.1cm]{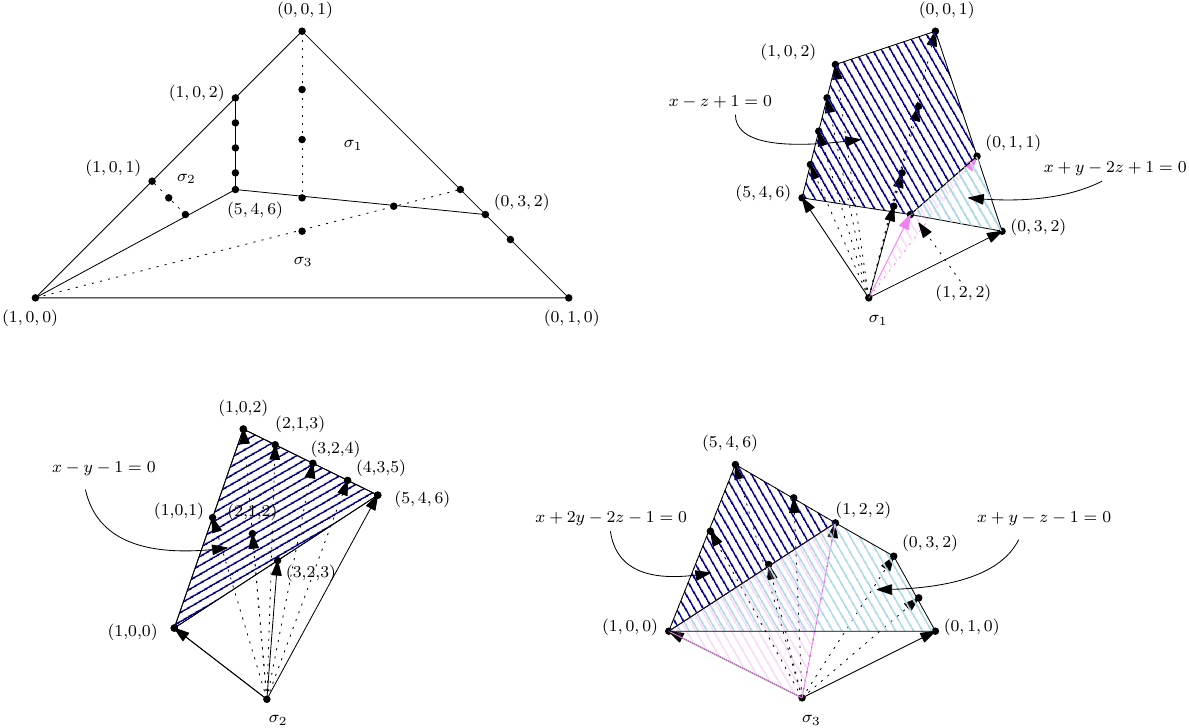}$  \caption{}  &
				
				\ \
				
				\ \
				
				$H_{\sigma_1}=\{(0,0,1),(1,0,2),(0,3,2),(5,4,6),(0,1,1),(1,1,2),(1,2,2),(2,1,3),(2,2,3),$
				
				\noindent $ (3,2,4),(3,3,5)(4,3,6)\}$ 
				
				$H_{\sigma_2}=\{(1,0,0),(1,0,2),(5,4,6),(1,0,1),(2,1,2),(2,1,3),(3,2,3),(3,2,4),(4,3,6)\}$
				
				$H_{\sigma_3}=\{(1,0,0),(0,0,1), (0,3,2),(5,4,6),(0,2,1),(1,1,1), (1,2,2),(3,2,3),(3,3,5)\}$
				\\
				\hline
				
			\end{tabular}
		\end{sideways}
	\end{table}

	\renewcommand*{\arraystretch}{0,5}
	\begin{table}[htbp]
		\begin{sideways}
			\begin{tabular}{ | p{5cm}  | p{6cm} | p{11cm} | }
				\hline
				{\bf Type of  $f$ } &  {\bf Subprofiles} &  {\bf EV $\sim$ Hilbert Basis} \\
				
				\hline
				\vskip.9cm $E_{0,7}:z^3+y^5+x^2y^2=0 $  
				
				\ \
				
				&
				\vskip.4cm
				$\includegraphics[width=5.5cm]{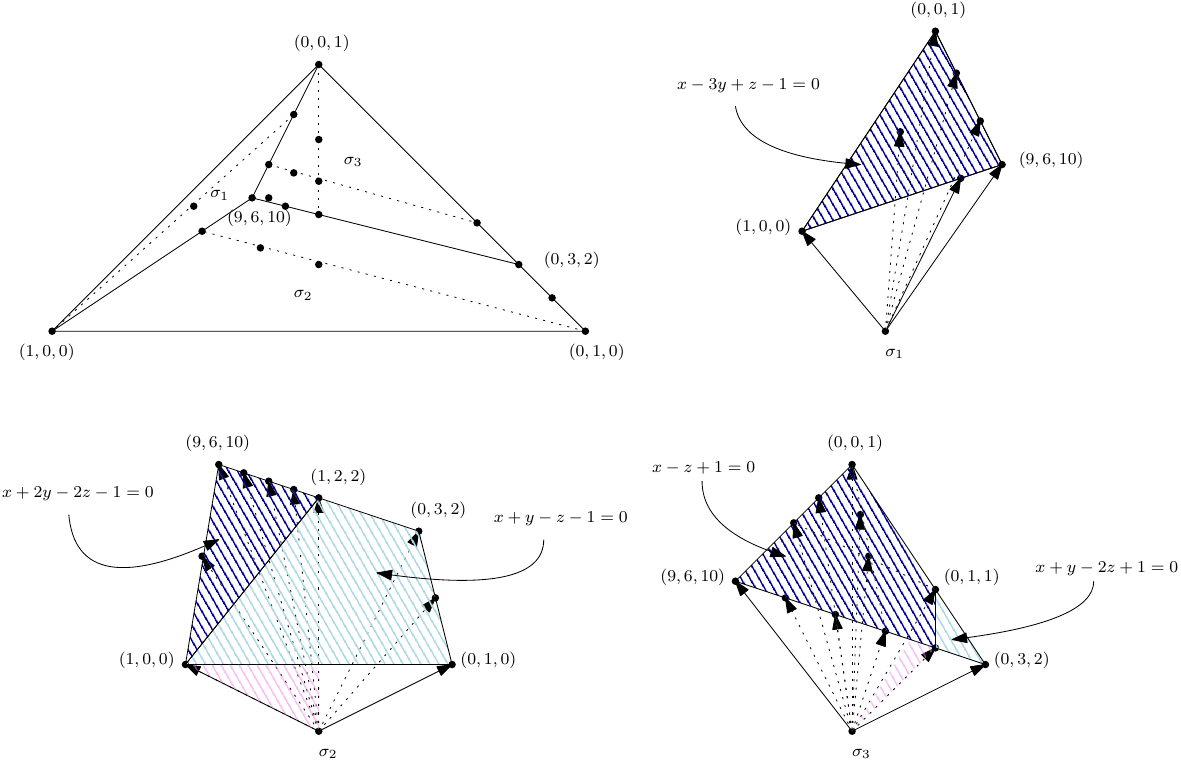}$   \caption{} &
				
				\ \
				
				\ \
				
				$H_{\sigma_1}=\{(1,0,0),(0,0,1),(9,6,10),(2,1,2),(3,2,4),(5,3,5),(6,4,7)\}$
				
				$H_{\sigma_2}=\{(1,0,0),(0,0,1),(0,3,2),(9,6,10),(0,2,1),(1,1,1),(1,2,2),(3,2,3),(3,3,4),$
				
				$(5,3,5),(5,4,6),(7,5,8)\}$
				
				$H_{\sigma_3}=\{(0,0,1),(0,3,2),(9,6,10),(0,1,1),(1,1,2),(1,2,2),(2,2,3),(3,2,4),(3,3,4),$
				
				$(4,3,5),(6,4,7),(5,4,6),(7,5,8)\}$
				
				\\
				\hline 
				\vskip.9cm $E_{7,0}:z^3+x^2yz+y^4 =0$
				
				\ \
				
				&
				\vskip.2cm
				$\includegraphics[width=5.5cm]{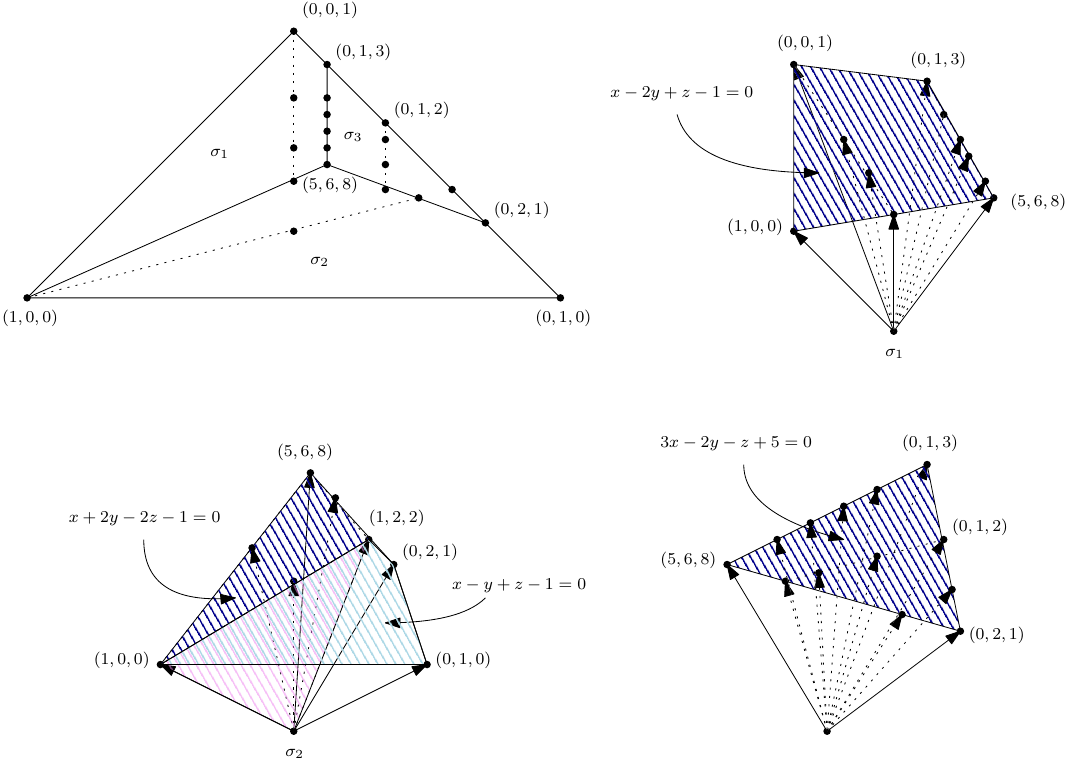}$   \caption{} & 
				
				\ \
				
				\ \
				
				$H_{\sigma_1}=\{(1,0,0),(0,0,1),(0,1,3),(5,6,8),(1,1,2),(1,2,4),(2,2,3),(2,3,5),(3,3,4),$
				
				$(3,4,6),(4,5,7)\} $
				
				$H_{\sigma_2}=\{(1,0,0),(0,1,0),(0,2,1),(5,6,8),(1,1,1),(1,2,2),(3,3,4),(3,4,5)\}$
				
				$H_{\sigma_3}=\{(0,1,3),(0,2,1),(5,6,8),(0,1,1),(0,1,2),(1,2,2),(1,2,3),(1,2,4),(2,3,4),$
				
				$(2,3,5),(3,4,5)\,(3,4,6),(4,5,7)\}$
				
				\\
				\hline
				
				\vskip.9cm $H_{n}:z^3+x^2y(x+y^{k-1})=0 $  
				
				\ \
				
				$n=3k-1$, $n\geq 1$
				
				\ \
				
				&
				\vskip.4cm
				$\includegraphics[width=5.5cm]{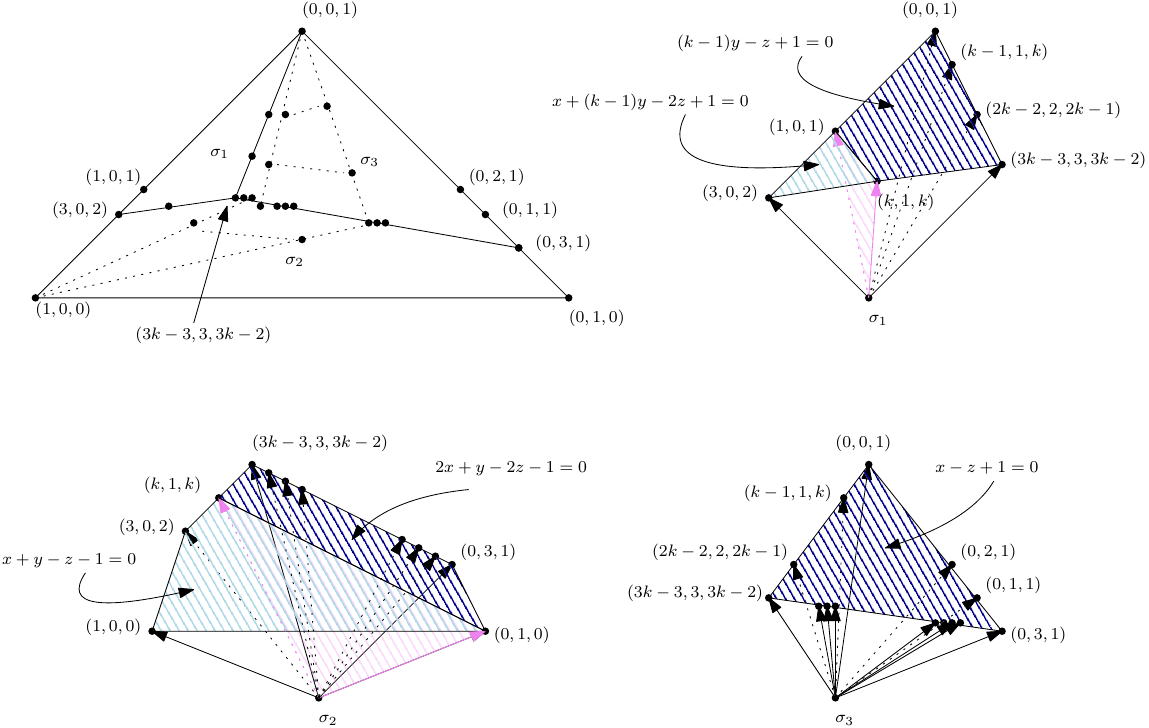}$   \caption{} &
				
				\ \
				
				\ \
				
				$H_{\sigma_1}=\{(0,0,1),(3,0,2),(3k-3,3,3k-2),(1,0,1),(k,1,k),(k-1,1,k),(2k-2,2,2k-1)\}$
				
				$H_{\sigma_2}=\{(1,0,0),(0,1,0),(0,3,1),(3k-3,3,3k-2),(2,0,1),(k,1,k),(1,3,2),(2,3,3),$
				
				$\ldots,(3k-4,3,3k-3),(1,1,1),(2,1,2),\ldots,(k-1,1,k-1)\}$
				
				$H_{\sigma_3}=\{(0,0,1),(0,3,1),(3k-3,3,3k-2),(0,1,1),(0,2,1),(k-1,1,k),(2k-2,2,2k-1),$
				
				$(1,3,2),(2,3,3),\ldots,(3k-4,3,3k-3),(1,1,2),(2,1,3),\ldots,(k-2,1,k-1),(2,2,3),$
				
				$(3,2,4),\ldots,(2k-4,2,2k-3),(1,2,2),(2,2,3),\ldots,(k-2,2,k-1)\}$
				\\
				\hline
			\end{tabular}
		\end{sideways}
	\end{table}
	
	\renewcommand*{\arraystretch}{0,5}
	\begin{table}[htbp]
		\begin{sideways}
			\begin{tabular}{ | p{5cm}  | p{6cm} | p{11cm} | }
				\hline
				{\bf Type of  $f$ } &  {\bf Subprofiles} &  {\bf EV $\sim$ Hilbert Basis} \\
				
				\hline
				\vskip.9cm $H_{n}:z^3+xy^kz+x^2y=0 $  
				
				\ \
				
				$n=3k$, $n\geq 1$ 
				
				\ \
				
				&
				\vskip.4cm
				$\includegraphics[width=6cm]{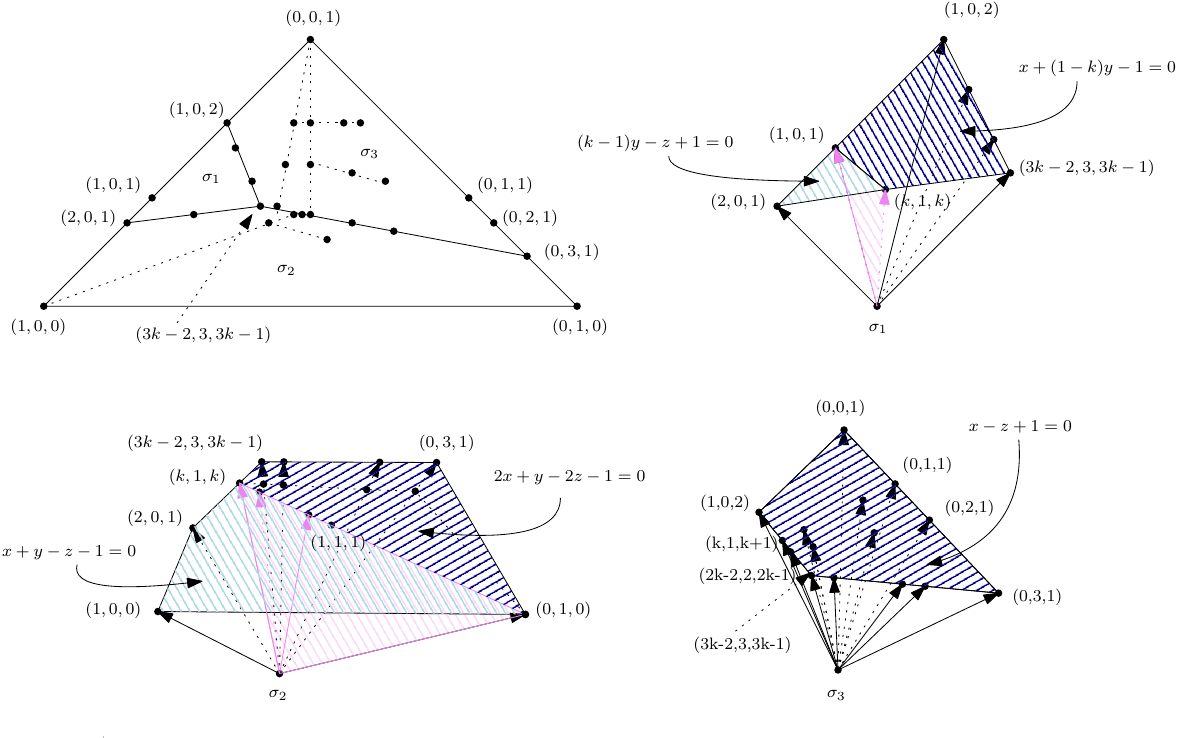}$   \caption{} & 
				
				\ \
				
				\ \
				
				$H_{\sigma_1}=\{(1,0,2),(2,0,1),(3k-2,3,3k-1),(1,0,1),(k,1,k),(k,1,k+1),(2k-1,1,2k)\}$
				
				$H_{\sigma_2}=\{(0,0,1),(1,0,2),(0,3,1),(3k-2,3,3k-2),(0,1,1),(0,2,1),(k,1,k+1),$
				
				$(2k-1,2,2k),(1,3,2),(2,3,3),(3,3,4),\ldots,(3k-6,3,3k-5),(3k-3,3,3k-2),(1,1,2),$
				
				$(2,1,3),\ldots,(k-2,1,k-1),(k-1,1,k),(2,2,3),(3,2,4),\ldots,(2k-4,2,2k-3),$
				
				$(2k-2,2,2k-1),(1,2,2),(2,2,3),\ldots,(2k-3,2,2k-2)\}$
				
				$H_{\sigma_3}=\{(1,0,0),(0,1,0),(2,0,1),(3k-2,3,3k-1),(1,3,2),(2,3,3),\ldots,(3k-3,3,3k-2),$
				
				$(1,1,1),(2,1,2),\ldots,(k,1,k)\}$
				
				\\
				\hline 
				\vskip.9cm $H_{n}: z^3+xy^{k+1}z+x^3y^2=0$
				
				\ \
				
				$n=3k+1$, $n\geq 1$
				
				\ \
				
				&
				\vskip.2cm
				$\includegraphics[width=6cm]{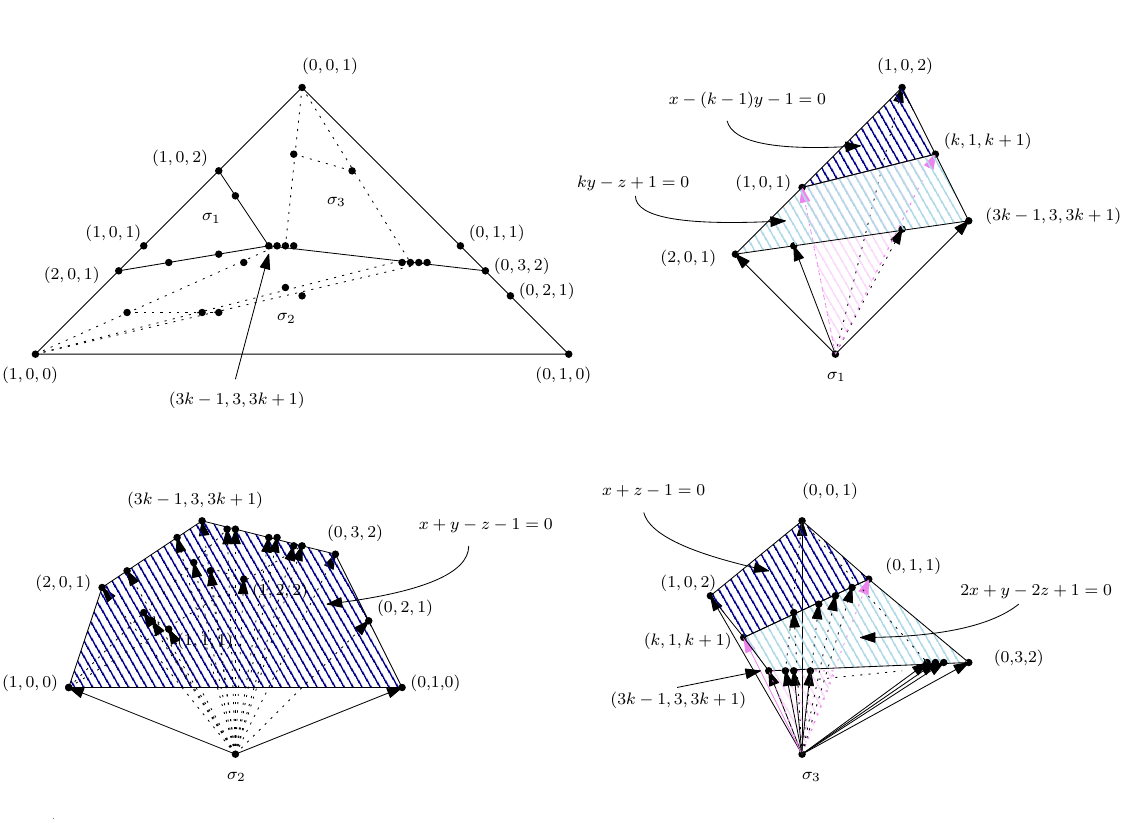}$  \caption{}  & 
				
				\ \ 
				
				\ \
				
				$H_{\sigma_1}=\{(1,0,2),(2,0,1),(3k-1,3,3k+1),(1,0,1),(k,1,k+1),(k+1,1,k+1),$
				
				$(2k,2,2k+1)\} $
				
				$H_{\sigma_2}=\{(1,0,0),(0,1,0),(0,3,2),(2,0,1),(3k-1,3,3k+1),(0,2,1),(1,3,3),(2,3,4),\ldots,$
				
				$(3k-2,3,3k),(1,1,1),(2,1,2),\ldots,(k+1,1,k+1),(1,2,2),(2,2,3),(3,2,4),\ldots,$
				
				$(2k,2,k+1)\}$
				
				$H_{\sigma_3}=\{(0,0,1),(1,0,2),(0,3,2),(3k-1,3,3k+1),(0,1,1),(1,3,3)(2,3,4),\ldots,$
				
				$(3k-2,3,3k),(1,1,2),(2,1,3),\ldots,(k,1,k+1)\}$
				
				\\
				\hline
			\end{tabular}
		\end{sideways}
	\end{table}
	
}

\newpage


\vskip.4cm

\noindent B. Karadeniz \c Sen \hfill C. Pl\'enat \\
Gebze Technical University  \hfill  Aix Marseille University, 12M, CMI\\ 
Department of Mathematics  \hfill   Technop\^ole Ch\^ateau-Gombert\\
41400, Kocaeli, Turkey  \hfill 39, rue F. Joliet Curie, 13453 Maresille Cedex 13 \\ 
E-mail: busrakaradeniz@gtu.edu.tr   \hfill E-mail: camille.plenat@univ-amu.fr\\

\noindent M. Tosun\\
Galatasaray University\\
Department of Mathematics\\  
Ortak{\"o}y 34357,  Istanbul, Turkey\\
E-mail: mtosun@gsu.edu.tr


\begin{thebibliography}{0}
	
	
	
	
	\bibitem{mag} A. Altinta\c{s} Sharland, G. \c{C}evik and M. Tosun, Nonisolated forms of rational triple singularities, Rocky Mountain J. Math., 46, No.2, (2016), 357-388.
	
	\bibitem{ACMZ} A. Altinta\c{s} Sharland, C. O. O\u{g}uz, M. Tosun and Z. Zafeirakopoulos,  Algorithm providing explicit equations of rational singularities, In preparation.
	
	\bibitem{AGS} F. Aroca, M. Gomez-Morales and K. Shabbir, Torical modification of Newton non-degenerate ideals, Revista de la Real Academia de Ciencias Exactas, Fisicas y Naturales. Serie A. Matematicas, 107-1, (2013), 221-239.
	
	\bibitem{ar-hu} F. Aroca, M. Gomez-Morales and H. Mourtada, Grobner fan and embedded resolutions of ideals on toric varieties, Beiträge zur Algebra und Geometrie/Contributions to Algebra and Geometry, (2023), 1-12.
	
	\bibitem{Artin} M. Artin, On isolated rational singularities of surfaces, Amer. J. Math. 88, (1966), 129-136.
	
	\bibitem{cg} C. Bouvier and G. Gonzalez-Sprinberg, Syst\'eme g\'en\'erateur minimal, diviseurs essentiels et G-d\'esingularisations de vari\'et\'es toriques, T\^{o}hoku Math. J., 47, (1995), 125-149.
	
	\bibitem{Fernex} T. de Fernex, The space of arcs of an algebraic variety, Algebraic Geometry: Salt Lake City 2015, Proc. Symp. Pure Math., Vol. 97-1, 2018.
	
	\bibitem{Fernex2} T. de Fernex, Three-dimensional countre-examples to the Nash problem, Compos. Math. Vol. 149, (2013), 1519-1534.
	
	\bibitem{dedo} T. de Fernex and R. Docampo, Terminal valuations and the Nash problem, Invent. Math. Vol. 203 (2016), 303-331.
	
	
	\bibitem{ELM} L. Ein, R. Lazarsfeld and M. Mustata, Contact loci in arc spaces,  Compos. Math. Vol. 140, (2004), 1229-1244.
	
	\bibitem{bobadilla-pe} J. Fernandez de Bobadilla and M. Pe Pereira, The Nash problem for surfaces, Annals of Math., 176, No.3, (2012), 2003-2029.
	
	\bibitem{Fukuda} K. Fukuda, A. N. Jensen and R. R. Thomas, Computing Gr\"{o}bner fans, Math. of Computation, 76, No.260, (2007), 2189-2212.
	
	\bibitem{giles} F. R. Giles and W. R. Pulleyblank, Total dual integrality and integer polyhedra, Linear Algebra and Its Applications, 25, (1979), 191-196.
	
	
	\bibitem{hironaka}  H. Hironaka, Resolution of singularities of an algebraic variety over a field of characteristic zero I, II,  Annals of Math., 79, No.1, (1964), 109-203.
	
	\bibitem{IK}  S. Ishii and J. Koll{\'a}r, The Nash problem on  arc families of singularities, Duke Math. J. 120-3, (2003), 601-620.
	
	
	\bibitem{Jensen} A. N. Jensen, Computing Grobner fans and tropical varieties in Gfan, The IMA Volumes in Math. and its appl., 148, (2008), 33-46.
	
	\bibitem{JK}  J. M. Johnson and J. Koll{\'a}r, Arc spaces of cA-type singularities, J. of Singularities, 7, (2013), 238-252.
	
	
	\bibitem{bhcm} B. Karadeniz, H. Mourtada, C. Pl\'enat and M. Tosun, The embedded Nash problem of birational models of rational triple singularities,  J. of Singularities, 22, (2020), 337-372.
	
	\bibitem{khovanskii} A. G. Khovanskii, Newton polyhedra (resolution of singularities), J. of Soviet Mathematics, 27, (1984), 2811-2830.
	
	\bibitem{kouch} A. G. Kouchnirenko, Polyédres de Newton et nombres de Milnor, Invent Math., 32, No.1, (1976), 1-31.
	
	\bibitem{LMR} M. Lejeune-Jalabert, H. Mourtada  and A. Reguera, Jet schemes and minimal embedded desingularization of plane branches, Revista de la Real Academia de Ciencias Exactas, Fisicas y Naturales. Serie A. Matematicas, 107-1, (2013), 145-157.
	
	\bibitem{Mora} T. Mora and L. Robbiano, The Gr\"{o}bner fan of an ideal, J. Symb. Comput., 6 (2/3), (1988), 183-203.
	
	\bibitem{Mo} H. Mourtada, Jet schemes of rational double point surface singularities,  Valuation Theory in Interaction, EMS Ser. Congr. Rep., Eur. Math. Soc., (2014), 373-388.
	
	\bibitem{hc}  H. Mourtada and C. Pl\'enat, Jet schemes and minimal toric embedded resolutions of rational double point singularities,  Comm. in Algebra, 46-3, (2018), 1314-1332.
	
	\bibitem{Nash} J. F. Nash, Arc structure of singularities, Duke Math. J., 81-1, (1995), 31-38.
	
	
	\bibitem{oka} M. Oka, Non-degenerate complete intersection singularity, Act. Math. Hermann, Paris, 1997.
	
	\bibitem{O1} M. Oka, On the resolution of the hypersurface singularities, Adv. Stud. Pure Math., 8, (1987), 405-436. 
	
	\bibitem{rosales} J. C. Rosales and P. A. Garcia-Sanchez, Finitely generated commutative monoids, Nova Science Publishers, Inc., New York, 1999.
	
	\bibitem{Tyurina} G. N. Tyurina, Absolute isolatedness of rational singularities and rational triple points, Fonc. Anal. Appl. 2-4, (1968), 324-332.
	
	\bibitem{Varc} A. N. Varchenko, Zeta-function of monodromy and Newtons diagram, Invent. Math., 37, No.3, (1976), 253-262.
	
\end{thebibliography}
\end{document}